\documentclass[11pt]{amsart}
\usepackage{amsmath,amssymb,amsthm}
\usepackage{hyperref}
\hypersetup{hidelinks=true}

\usepackage{bbm}

\usepackage[utf8]{inputenc}

\usepackage{graphicx}
\usepackage{enumerate}
\usepackage{amsfonts,amssymb,latexsym}
\usepackage{accents}
\usepackage{afterpage}
\usepackage[style=alphabetic,maxbibnames=99, backend=bibtex]{biblatex}

\newcommand{\ep}{\varepsilon}
\newcommand{\N}{\mathbb{N}}

\newcommand{\R}{\mathbb{R}}

\newcommand\restr[2]{{
  \left.\kern-\nulldelimiterspace 
  #1 
  \vphantom{\big|} 
  \right|_{#2} 
  }}

\DeclareMathOperator{\co}{co}

\renewcommand{\geq}{\geqslant}
\renewcommand{\leq}{\leqslant}

\newtheorem{theorem}{Theorem}[section]
\newtheorem{lemma}[theorem]{Lemma}

\newtheorem{prop}[theorem]{Proposition}
\newtheorem{corollary}[theorem]{Corollary}
\theoremstyle{definition}
\newtheorem{definition}[theorem]{Definition}

\theoremstyle{remark}
\newtheorem{remark}[theorem]{Remark}
\numberwithin{equation}{section}

\def\fnote#1{\footnote}

\def\R{{\mathbb R}}

\def\ignora#1{}
\def\n3#1{\left\vert  \! \left\vert \! \left\vert \, #1 \, \right\vert \!
  \right\vert \! \right\vert }

\include{draft}

\begin{document}



\title[Compact retracts]{Compact retractions and Schauder decompositions in Banach spaces}

\author{Petr H\'ajek}\thanks{This research was supported by CAAS CZ.02.1.01/0.0/0.0/16-019/0000778
 and by the project  SGS21/056/OHK3/1T/13.}
\address[P. H\'ajek]{Czech Technical University in Prague, Faculty of Electrical Engineering.
Department of Mathematics, Technická 2, 166 27 Praha 6 (Czech Republic)}
\email{ hajek@math.cas.cz}

\author{ Rub\'en Medina}\thanks{The second author research has also been supported by MICINN (Spain) Project PGC2018-093794-B-I00 and MIU (Spain) FPU19/04085 Grant.}
\address[R. Medina]{Universidad de Granada, Facultad de Ciencias.
Departamento de An\'{a}lisis Matem\'{a}tico, 18071-Granada
(Spain); and Czech technical University in Prague, Faculty of Electrical Engineering.
Department of Mathematics, Technická 2, 166 27 Praha 6 (Czech Republic)}
\email{rubenmedina@ugr.es}
\urladdr{\url{https://www.ugr.es/personal/ae3750ed9865e58ab7ad9e11e37f72f4}}

\maketitle 

\begin{abstract}

In our note we show the very close connection between the existence of 
a Finite Dimensional Decomposition (FDD for short) for a separable Banach space $X$
and the existence of a Lipschitz retraction of $X$ onto a small (in a certain precise sense) generating convex and compact
subset $K$ of $X$.

In one direction,  if $X$ admits an FDD  then we construct a Lipschitz retraction onto
a small generating convex and compact set $K$. On the other
hand, we prove that if $X$ admits a small generating compact Lipschitz retract then $X$ has the $\pi$-property.
We note that it is still unknown if the $\pi$-property is isomorphically equivalent to the existence of an FDD.

For dual Banach spaces this is true, so our results lead in particular to a characterization of the FDD property for dual Banach spaces $X$
in terms of the existence of Lipschitz retractions onto small generating convex and compact subsets of $X$.

It is conceivable that our results will find applications in the area of Lipschitz isomorphisms of Banach spaces.

Our arguments make critical use of the Lipschitzization of coarse Lipschitz mappings due to J. Bourgain, and of
an unpublished complementability result of V. Milman.

 We  give an example of a small generating convex  compact set which is not a Lipschitz retract of $C[0,1]$,
although it is contained in a small convex Lipschitz retract and contains another one. 

In the last part of our note we characterize isomorphically Hilbertian spaces as those Banach spaces $X$ for which every convex and compact subset is a Lipschitz retract of $X$. 
Finally, we prove that a convex and compact set $K$ in any Banach space with
a Uniformly Rotund in Every Direction norm is a uniform retract, of every bounded set containing it, via the nearest point map.
\end{abstract}

\section{Introduction}

The study of retractions is a large area of research in topology and nonlinear analysis with many applications.
In the uniform or Lipschitz setting  the authoritative
monograph \cite{BL2000} contains many fundamental results and the general point of view.

The results in our paper were originally motivated by the natural question asked by Godefroy and Ozawa in \cite{GO14}
(and then subsequently in \cite{God15}, \cite{God215}, \cite{God20},  and \cite{GP19}) whether every
separable Banach space admits a generating convex and compact Lipschitz retract (GCCR, for short).
By a generating  set in a Banach space $X$ we mean the set $C$ such that the closed linear span of $C$ is the whole space $X$.
We will say that $C$ is a $\lambda$-GCCR if it is a GCCR and there exists a $\lambda$-Lipschitz retraction from $X$ onto $C$.
A positive answer to this problem would immediately imply that every separable Banach space is
Lipschitz approximable, solving an important open problem of Kalton.
Using a rather short, but ingenious argument, Kalton \cite{Kal12} showed that indeed if $X$ has a separable dual
(or it is itself a separable dual) then $X$
is approximable. But his compact identity-approximating mappings are far from being retractions. In \cite{GO14}
it is noted that the retractions can be constructed rather easily if $X$ has an unconditional basis, but the same approach 
fails even for spaces with a Schauder basis.

Our first result (Theorem \ref{theoFDDcompact}) is that a Banach space $X$  with an FDD admits a GCCR. Moreover, if
$X$ has a monotone FDD, then it has a  $(5+\epsilon)$-GCCR, for every $\epsilon>0$. If
$X$ has a monotone Schauder basis, then it has a  $(1+\epsilon)$-GCCR, for every $\epsilon>0$. 
Our GCCR is roughly speaking a diamond shaped set, and the
 retraction mapping is somehow aligned with the 
canonical projections but it is certainly not the nearest point map.  

The compact sets used in our arguments
are in some sense small. Later on, for the purposes of proving  results going in the opposite direction, we proceed by formally defining the
quantitative  concept of a small set, which is intuitively a generating compact
subset of $X$  that  is contained in sufficiently small neighbourhoods of its finite dimensional sections.
The smallness condition restricts the asymptotic behaviour of the compact set, in a certain sense, but  otherwise it leaves a
 complete freedom as regards its possible shape.  It is curious that  our  results in both directions
lead to small compacts of roughly the same size.

 Our second, and perhaps the main result of the note (Theorem \ref{theosmall}),  is that
if a Banach space admits a small Lipschitz retract (in particular, a small  GCCR)  then it has the $\pi$-property.
Note that a Banach space with the metric $\pi$-property has an FDD \cite{JW70}, see also \cite{Cas01} Thm. 6.4, 
and it is still open if the $\pi$-property implies metric $\pi$-property (and hence FDD) under an equivalent renorming.
So our previous results combined together are possibly a characterization of  the existence of an FDD property for the Banach space $X$  in terms of 
the existence of Lipschitz retractions onto small GCCR. In the case of dual spaces the characterization holds true, thanks to Theorem 1.3 in \cite{Joh2}.

The proof is based
on several deep ingredients. Of course, one would like to use the (Gateaux) differentiation theory for Lipschitz
mappings in separable Banach spaces in order to pass from the compact Lipschitz retraction to
its linearization with a small range. However,  to do this directly using the 
abundance of points of Gateaux differentiability  does not seem possible. Instead, our proof takes a detour, and produces the finite rank
linear projections (needed for the $\pi$-property) indirectly, only proving their existence. 
We first pass from the small retract $K$ to a finite dimensional subspace $E_n$, which contains the bulk of the points of $K$, in a certain sense. Of course, we do not immediately have in our hands
a good Lipschitz retraction from $X$ onto $E_n$, but using the "Lipschitzization" of coarse Lipschitz mapings due to Bourgain \cite{Bou87} (in the formulation of Begun \cite{B99}) we obtain a good Lipschitz almost retraction to $E_n$ from  finite dimensional
subspaces $G$ of $X$, of controlled dimension, which contain $E_n$. For the next step we use the deep and yet unpublished result,
 due to Vitali Milman,
communicated to us by Bill Johnson. Namely, the projection constant $\lambda(E_n,X)$
of a $n$-dimensional subspace $E_n$ of $X$ is witnessed by $\lambda(E_n, G)$ for a certain finite dimensional subspace $G$ of $X$,
of dimension roughly $3^n$. This unexpected fact makes our subsequent argument  quantitatively independent of the
Banach space $X$.  At this point we may use the differentiation theory and averaging in the finite dimensional setting, applied to
the finite dimensional approximate version of our retraction (crucially using the smallness assumption for $K$)
to produce a good linear projection from $G$ to $E_n$. But this implies that $\lambda(E_n,X)$ cannot be too large, which eventually yields the $\pi$-property for $X$.

Since the $\pi$-property is equivalent to the FDD property for dual Banach spaces \cite{JRZ71}, as a corollary to our
above results we obtain a retractional characterization of  an FDD in the class of dual Banach spaces.

We remark that our techniques above are applicable to small compacts only, and we do not know
if analogous results hold for general convex compact subsets of $X$. 

Using the same approach,
we also give a variant of the result of Godefroy and Kalton in \cite{GK03}, who characterized the BAP
property by means of a sequence of finite rank Lipschitz mappings pointwise convergent to the identity. Namely, using the finite rank Lipschitz retractions we characterize the $\pi$-property.

We also note that
if there is a Lipschitz retraction of $X$ onto a convex compact set $K$, and $Y$ is a closed linear subspace of $X$
spanned by $K$, which is linearly complemented in its bidual $Y^{**}$, then $Y$ is a complemented subspace of $X$.
This means that a natural way of getting the Lipschitz retraction from $X$ onto $K$ is  to simply 
compose the linear projection from $X$ onto $Y$ with a Lipschitz retraction from $Y$ onto $K$.

In the next section, we give an example of a small  compact convex set $K$ in the space $C[0,1]$, which is contained in a small GCCR,
 and contains another small GCCR, but such that $K$ is not a 
GCCR for the space $C[0,1]$. This example underlines the subtleness of the existence of Lipschitz retractions.

In our final section we first give a new characterization of isomorphically Hilbertian spaces as those Banach spaces for which every convex and compact subset
is a Lipschitz retract. 

Then we
 proceed to the problem of uniformly continuous retractions onto convex compact sets.
Our main result is that if a  Banach space is equipped with a URED norm, then
every convex  and compact subset is a uniformly continuous retract, from any bounded superset, with respect to the nearest point map.
We recall that every separable Banach space admits an equivalent URED renorming.
This result implies, in particular, that every convex and compact set is an absolute uniform retract,
a fact recently established in \cite{CCW21}. The case of a general (nonseparable) space $X$ is also treated.

\section{Preliminaries}

Our notation and terminology is standard. We work in the setting of real Banach spaces throughout this note.
For the general concepts and results of Banach space theory we refer to \cite{Fab1}.
The background on the various forms of the approximation property can be found in \cite{Cas01}.
For  the theory of Lipschitz and uniform retracts we refer to the first
two chapters in the monograph \cite{BL2000}.

Let us now pass to some definitions and results which are used heavily in our note. We start with the formulation of  several classical
concepts related to the approximation property, in the growing generality.

A Schauder basis for a real Banach space $X$ is a sequence $(x_n)\subset X$ with the property that for every $x\in X$, there exists a unique sequence $(\alpha_n)\subset \R$ such that
$$\bigg|\bigg| x-\sum\limits_{i=1}^n\alpha_ix_i \bigg|\bigg|\xrightarrow{n\to\infty}0.$$
In this case we say that $X$ has a Schauder basis  and we call the projections $P_n(x)=\sum\limits_{i=1}^n\alpha_ix_i$ the natural projections of the Schauder basis.
If $(x_n)$ is a Schauder basis in a Banach space $X$, then we denote by  $(x^*_n)$ the coordinate functionals, which form a 
basic sequence in the dual space $X^*$.

\begin{definition}
A sequence $(E_n)$ of finite dimensional subspaces of a Banach space $X$ is called the finite dimensional decomposition (FDD for short) if for every $x\in X$ there is a unique sequence $x_n\in E_n$ so that
$$\bigg|\bigg| x-\sum\limits_{i=1}^nx_i \bigg|\bigg|\xrightarrow{n\to\infty}0.$$
In this case we say that $X$ has an FDD and we call the projections $P_n(x)=\sum\limits_{i=1}^nx_i$ the natural projections of the FDD.
\end{definition}

If $X$ is a Banach space with a Schauder basis (resp. FDD) then the natural projections of the basis (resp. FDD) are uniformly bounded. Moreover, $X$ can be equivalently renormed so that this uniform bound is 1. In this case we say that the Schauder basis (resp. FDD) is monotone.

\begin{definition}
A Banach space $X$ is said to have the $\pi$-property if there is a uniformly bounded net of finite rank projections $(S_\alpha)$ on $X$ converging strongly to the identity on $X$. If this uniform bound is $\lambda\ge1$ then we  say that $X$ has the $\pi_\lambda$-property. 
In the case when  $\lambda=1$ we say that $X$ has the metric $\pi$-property.
\end{definition}

\begin{prop}[\cite{Cas01}]\label{conmutingpi}
For a separable Banach space $X$ and $\lambda\ge1$, the following are equivalent:
\begin{itemize}
\item X has the $\pi_\lambda$-property
\item There is a sequence of $\lambda$-bounded finite rank projections $(S_n)$ on $X$ pointwise converging to the identity such that
$$S_mS_n=S_n \;\;\;\;\forall m\ge n.$$
\end{itemize}
\end{prop}

As we mentioned above, it seems to be an open problem whether the $\pi$-property is equivalent to the existence of an FDD for a separable Banach space.
\begin{definition}
Let $X$ be a Banach space. If there is a uniformly bounded net of finite rank operators $(T_\alpha)$ on $X$ tending strongly to the identity on $X$, then we say that $X$ has the bounded approximation property (BAP for short). If $\lambda\ge1$ is a uniform bound for the net then we say that $X$ has the $\lambda$-bounded approximation property ($\lambda$-BAP for short). We refer to  the 1-BAP as the metric approximation property (MAP for short).\\
\end{definition}

\begin{definition}
Let $X$ be a Banach space. If there is a net of finite rank operators $(T_\alpha)$ on $X$ converging to the identity on $X$ uniformly on compacta, then we say that $X$ has the approximation property (AP for short).\\
\end{definition}

\begin{definition}
A Banach space $X$ is said to have the compact approximation property if for every $\ep>0$ and every compact set $K$ in $X$ there is a compact operator $T\in \mathcal{L}(X)$ so that $||Tx-x||\le \ep$ for all $x\in K$.
\end{definition}

For an arbitrary Banach space $X$, the previously defined concepts are ordered from the strongest to the weakest, that is,
$$\text{Schauder basis }\Rightarrow\text{ FDD }\Rightarrow\;\pi\text{-property }\Rightarrow\text{ BAP }\Rightarrow\text{ AP }\Rightarrow\text{ CAP.}$$
With the possible exception of $\text{ FDD }\Rightarrow\;\pi\text{-property }$, none of the above implications can be reversed (\cite{Sza87},\cite {Rea}, \cite{FJ73} and \cite{Wil92}).

Let us pass to  the non-linear approximation properties.

\begin{definition}
A map $T$ from a metric space $M$ into another metric space $N$ is said to be Lipschitz if there exists some $\lambda>0$ such that
$$d(T(x),T(y))\le \lambda d(x,y) \;\;\;\;\forall x,y\in M.$$
We say that $\lambda$ is the Lipschitz constant for $T$ and we call the infimum of all Lipschitz constants for $T$ the Lipschitz norm of $T$, that is,
$$||T||_{Lip}=\inf\big\{\lambda>0,\text{ Lipschitz constant for }T\big\}=\sup\limits_{x,y\in M, x\ne y}\frac{d\big(T(x),T(y)\big)}{d(x,y)}.$$
If $\lambda>0$ is a Lipschitz constant for $T$ then we say that $T$ is $\lambda$-Lipschitz.
\end{definition}

\begin{definition}
Let $X$ be a Banach space. If there is a net of finite rank Lipschitz maps $(T_\alpha)$ on $X$, whose Lipschitz norms are uniformly bounded, converging uniformly on compacta to the identity on $X$, then we say that $X$ has the Lipschitz bounded approximation property. If this net is bounded by $\lambda\ge1$ then we  say that $X$ has the $\lambda$-Lipschitz bounded approximation property.
\end{definition}

\begin{theorem}[\cite{GK03} Theorem 5.3]\label{GODKAL}
Let $X$ be an arbitrary Banach space. Then the following conditions are equivalent:
\begin{itemize}
\item X has the $\lambda$-BAP.
\item The Lipschitz free space $\mathcal{F}(X)$ has the $\lambda$-BAP.
\item X has the $\lambda$-Lipschitz bounded approximation property.
\end{itemize}
\end{theorem}

\begin{definition}
We say that a complete metric space $M$ is approximable whenever there is a subadditive map $\omega:[0,\infty)\rightarrow[0,\infty)$ with $\lim\limits_{t\to0}\omega(t)=0$ so that for every finite set $E\subset M$ and every $\ep>0$ we can find a uniformly continuous map $\psi:M\rightarrow M$ such that $d(x,\psi(x))<\ep$ for every $x\in E$, $\psi(M)$ is relatively compact and the modulus of continuity of  $\psi$ is bounded by $\omega$.\\
If $\omega(t)=Lt$ for some $L>0$ then we say that $M$ is Lipschitz approximable.
\end{definition}

If $M$ is separable then it is easy to see that $M$ is approximable (resp. Lipschitz approximable) if and only if there is an equi-uniformly continuous (resp. equi-Lipschitz) sequence of maps $\psi_n:M\rightarrow M$ with relatively compact range such that $\lim\limits_{n\to\infty}d(x,\psi_n(x))=0$ for every $x\in X$.

These concepts were  introduced and studied in several papers, \cite{Kal04}, \cite{Kal12} and \cite{GLZ14}.

Kalton proved in \cite{Kal12} that every Banach space with a separable dual (or a separable dual space itself) is
approximable. It is still an open problem if every separable Banach space is approximable.
Godefroy, on the other hand, observed in \cite{God20} that the compact approximation property implies that the space is
Lipschitz approximable, showing  that  the Lipschitz approximability is a strictly weaker property than the AP.\\

In the next sections, we are going to make a repeated use of the following concepts.

\begin{definition}
Let $M$ be a metric space and $N\subset M$. A retraction from $M$ onto $N$ is a map $R:M\rightarrow N$ such that $\restr{R}{N}=Id_N$.
 In this case we say that $N$ is a retract of $X$. If $R$  is Lipschitz (resp. uniformly continuous) then we say that $R$ is a Lipschitz (resp. uniformly continuous) retraction onto $N$ and $N$ is a Lipschitz (resp. uniformly continuous) retract of $M$.

If $N$ is a Lipschitz (resp. uniformly continuous) retract of $M$ for every metric space $M$ containing it, we say that $N$ is an absolute Lipschitz (resp. uniformly continuous) retract.
\end{definition}

We will say that a subset $K$ of a Banach space $X$ generates $X$ whenever the closed linear span of $K$ is equal to $X$. A Lipschitz retract $K$ of a Banach space $X$ that is convex, compact, and generates $X$ is going to be called a generating convex compact retract (GCCR for short) of $X$.

Note that the existence of a GCCR implies that the space $X$ is Lipschitz approximable.

\begin{definition}
Let $M$ be a metric space,  $N\subset M$. We say that $R:M\rightarrow N$  is a proximity map (or a nearest point map) onto $N$ if
$$d(R(x),x)=\inf\limits_{y\in N}d(y,x)\;\;\;\forall x\in M.$$
\end{definition}

A proximity map may not exist in some situations and if it exists it may not be unique. If $M$ is a uniformly convex Banach space then this map is known to be unique and uniformly continuous whenever $N$ is a closed convex subset \cite{Bjo79}.

\begin{definition}
A normed space $X$ is said to be uniformly rotund in the direction $z\in X$ if whenever $(x_n)$ and $(y_n)$ are two sequences in $X$ such that
\begin{enumerate}
\item $||x_n||=||y_n||=1$ for every $n\in\N$,
\item $\lim\limits_{n\to\infty}\big|\big| \frac{x_n+y_n}{2} \big|\big|=1$,
\item There is a sequence of real numbers $(r_n)$ such that $x_n-y_n=r_nz$ for every $n\in\N$,
\end{enumerate}
then $\lim\limits_{n\to\infty}||x_n-y_n||=0$.\\
If $X$ is uniformly rotund in the direction $z$ for every $z\in S_X$ then we say that $X$ is uniformly rotund in every direction (URED for short).
\end{definition}

\section{Compact Lipschitz retracts}\label{secretr}

The problem whether every separable Banach space has a GCCR
was posed in \cite{GO14}. In fact, having a GCCR for a Banach space $X$ gives a lot of information. For instance, it implies that $X$ is Lipschitz approximable and that $X$ has the BAP if and only if $\mathcal{F}(K)$ has the BAP by Theorem \ref{GODKAL}. Our goal in this section is a construction of a certain diamond shaped convex and compact set $K$, which generates the Banach space $X$, together with a Lipschitz retraction from $X$ onto $K$. Our construction is performed
under the assumption that $X$ has a FDD so in particular it gives a positive answer to the question asked in \cite{GP19}, whether Pelczynski's space has a GCCR. In order to make the construction work, we
are forced to make $K$ in some sense small. The method breaks down if we try to replace the FDD by
a weaker concept, e.g. a Markushevich basis or so. As we will find out in the subsequent parts of
our note, there is a good reason for that. Namely, the kind of retractions we are using here
imply that $X$ has the $\pi$-property (which is possibly equivalent to having an FDD).

Our approach is to define some Lipschitz retractions onto increasing finite-dimensional sections of $K$, so that then we proceed by taking a limit of these mappings to define the final retraction onto $K$. All these initial retractions will be defined as compositions of little perturbations of the natural projections of the FDD. To this end, we first prove some Lemmas stating the Lipschitz behaviour of the perturbed projections.

Let us proceed with the construction.
Let $(X,||\cdot||)$ be a Banach space with a monotone FDD  $(X_n)_{n\in\N}$ whose natural projections are $(P_n)_{n\in\N}$. For every $x\in X$ and $i\in\N$ we are going to denote $x_i=(P_i-P_{i-1})(x)$ where $P_0\equiv 0$. Let $(r_n)_{n\in \N}$ be an arbitrary decreasing sequence of positive real numbers. For every $m\in\N$  we define the function $f_m:X\rightarrow \R,$ by setting $f_1:X\rightarrow\{r_1\}$ the constant function, and
$$f_m(x)=r_m\left(1-\sum\limits_{i=1}^{m-1}\frac{||x_i||}{r_i}\right),\; m\ge2.$$
We will define for $n\in\N$ the following subsets of $X$
$$K=\overline{\co}\left(\bigcup\limits_{k\in\N}r_kB_{X_k}\right)\;\;\;\;,\;\;\;\;K_n=\overline{\co}\left(\bigcup\limits_{k=1}^nr_kB_{X_k}\right).$$
Finally, given $n,m\in\N$ such that $m\leq n$ we consider $E_{n,m}=P_m(X)\cup K_n$, and $F_{n,m}:E_{n,m}\rightarrow E_{n,m-1}$ given by
$$F_{n,m}(x)=\begin{cases}x\;,\;&\text{if }\;\sum\limits_{i=1}^{m}\frac{||x_i||}{r_i}\leq1,\\
P_{m-1}(x)\;,\;&\text{if }\;\sum\limits_{i=1}^{m-1}\frac{||x_i||}{r_i}\geq1,\\
P_{m-1}(x)+\frac{x_m}{||x_m||}f_m(x)\;,\;&\text{if }\;\sum\limits_{i=1}^{m-1}\frac{||x_i||}{r_i}<1<\sum\limits_{i=1}^{m}\frac{||x_i||}{r_i},\end{cases}$$
where, if $m=1$, we consider $\sum\limits_{i=1}^{0}\frac{||x_i||}{r_i}=0$.

One may see every $F_{n,m}$ as a perturbation of $P_{m-1}$ restricted to $E_{n,m}$. The next Lemma \ref{lemmalip2} states that it is possible to approximate the Lipschitz behaviour of $F_{n,m}$ by $F_{n-1,m}$.

\begin{lemma}\label{lemmalip2}
There exists a sequence $(A_m)_{m\in\N}\subset \R^+$ such that 
$$f_m\;\text{ is }\;\frac{r_mA_{m-1}}{r_{m-1}}\text{-Lipschitz}\;\;\;\;\;\forall m>1,$$
and if $m,n\in\N$ such that $m\le n-1$ then,
$$||F_{n,m}(x)-F_{n,m}(y)||\le\frac{r_nA_n}{r_{n-1}}||x-y||+||F_{n-1,m}\circ P_{n-1}(x)-F_{n-1,m}\circ P_{n-1}(y)||,$$
for every $x,y\in E_{n,m}$.
\begin{proof}
Just consider $A_m=2m$. Then, 
$$||P_m(x)||_{\ell_1(\{X_i\}_{i=1}^m)}\le A_m ||P_m(x)||\;\;\;\;\forall x\in X,$$
so we have that
$$\begin{aligned}|f_m(x)-f_m(y)|&=\left|r_m\left(\sum\limits_{i=1}^{m-1}\frac{||x_i||-||y_i||}{r_i}\right)\right|\le \frac{r_m}{r_{m-1}}\sum\limits_{i=1}^{m-1}||x_i-y_i||\\
&\le \frac{r_mA_{m-1}}{r_{m-1}}||P_{m-1}(x-y)||\le  \frac{r_mA_{m-1}}{r_{m-1}}||x-y||.\end{aligned}$$
Now, to prove the second part of the lemma, let us define $G_{n-1,m}:E_{n,m}\rightarrow E_{n-1,m-1}$ by
$$G_{n-1,m}(x)=F_{n-1,m}(P_{n-1}(x))\;\;\;\;\;\forall x\in E_{n,m}.$$
Taking into account the previous definitions, it is immediate that
$$G_{n-1,m}(x)=F_{n,m}(P_{n-1}(x))\;\;\;\;\;\forall x\in E_{n,m}.$$
Hence, it holds that $g_n:=F_{n,m}-G_{n-1,m}=P_{n}-P_{n-1}$.\\
Now, if $x\neq y\in E_{n,m}$ are such that $\sum\limits_{i=1}^{m}\frac{||x_i||}{r_i}\le1$ and $\sum\limits_{i=1}^{m}\frac{||y_i||}{r_i}\le 1$ then
$$\frac{||F_{n,m}(x)-F_{n,m}(y)||}{||x-y||}=1.$$
If that is not the case, then we may assume that $\sum\limits_{i=1}^{m}\frac{||y_i||}{r_i}>1$, so $||y_n||=0$ and $F_{n,m}(y)=F_{n,m}(P_{n-1}(y))=G_{n-1,m}(y)$. It turns out that
$$\begin{aligned}\frac{||F_{n,m}(x)-F_{n,m}(y)||}{||x-y||}&\le \frac{||g_n(x)||+||G_{n-1,m}(x)-G_{n-1,m}(y)||}{||x-y||}\end{aligned}.$$
It is enough to prove that $\frac{||g_n(x)||}{||x-y||}\le\frac{r_nA_n}{r_{n-1}}$. If $\sum\limits_{i=1}^{n}\frac{||x_i||}{r_i}>1$ then $||x_n||=0$ and it is obviously true. Otherwise,
$$\begin{aligned} \frac{||x_n||}{r_n}\le1- \sum\limits_{i=1}^{n-1}\frac{||x_i||}{r_i}\le \sum\limits_{i=1}^{n-1}\frac{||y_i||-||x_i||}{r_i}\le \sum\limits_{i=1}^{n-1}\frac{||x_i-y_i||}{r_i}\le\frac{\sum\limits_{i=1}^{n-1}||x_i-y_i||}{r_{n-1}} .\end{aligned}$$
So it is true that
$$||x-y||_{\ell_1(\{X_i\}_{i=1}^n)}\ge ||x_n||\frac{r_{n-1}}{r_n}.$$
Finally,
$$\frac{||g_n(x)||}{||x-y||}\le \frac{||x_n||A_n}{||x-y||_{\ell_1(\{X_i\}_{i=1}^n)}}\le\frac{||x_n||A_n}{||x_n||\frac{r_{n-1}}{r_n}}\le \frac{r_nA_n}{r_{n-1}}.$$
\end{proof}
\end{lemma}

\begin{prop}\label{proplip2}
For every $j\in\N$ and $x,y\in P_j(X)$ the following implications hold true:
\begin{equation}\label{1ineqFDD}\text{if }\sum\limits_{i=1}^{j-1}\frac{||x_i||}{r_i},\sum\limits_{i=1}^{j-1}\frac{||y_i||}{r_i}\le1\Rightarrow||F_{j,j}(x)-F_{j,j}(y)||\le \left( 5+\frac{r_jA_{j-1}}{r_{j-1}} \right)||x-y||,\end{equation}
\begin{equation}\label{2ineqFDD}\text{if }\sum\limits_{i=1}^{j-1}\frac{||y_i||}{r_i}>1\Rightarrow||F_{j,j}(x)-F_{j,j}(y)||\le \left( 1+\frac{r_jA_{j-1}}{r_{j-1}} \right)||x-y||.\end{equation}
\begin{proof}
First we prove \eqref{1ineqFDD} case by case:

If $\sum\limits_{i=1}^{j}\frac{||x_i||}{r_i}>1$ and $\sum\limits_{i=1}^{j}\frac{||y_i||}{r_i}>1$, then from the definition of $f_j$ we know that $f_j(y)<||y_j||$ so
$$\begin{aligned}&\frac{\big|\big|x_j||y_j||f_j(x)-y_j||x_j||f_j(y) \big|\big|}{||x_j||\,||y_j||}\\
&\le\frac{\big|\big|x_j||y_j||f_j(x)-x_jf_j(y)||x_j||\,\big|\big|+\big|\big|x_jf_j(y)||x_j||-y_j||x_j||f_j(y) \big|\big|}{||x_j||\,||y_j||}\\
&=\frac{\big|f_j(x)||y_j||-f_j(y)||x_j||\,\big|}{||y_j||}+\frac{||x_j-y_j||f_j(y)}{||y_j||}\\
&\le\frac{\big|f_j(x)||y_j||-f_j(y)||y_j||\,\big|+\big|f_j(y)||y_j||-f_j(y)||x_j||\,\big|}{||y_j||}+||x_j-y_j||\\
&=|f_j(x)-f_j(y)|+\frac{f_j(y)||x_j-y_j||}{||y_j||}+||x_j-y_j||\\
&\le \frac{r_jA_{j-1}}{r_{j-1}}||x-y||+2||x_j-y_j||\le\left(4+\frac{r_jA_{j-1}}{r_{j-1}}\right)||x-y||.\end{aligned}$$
Hence, we have that
$$\begin{aligned} &||F_{j,j}(x)-F_{j,j}(y)||\\
&\le ||P_{j-1}(x-y)||+||(F_{j,j}-P_{j-1})(x)-(F_{j,j}-P_{j-1})(y)||\\
&=||P_{j-1}(x-y)||+\left|\left| \frac{x_j}{||x_j||}f_j(x)-\frac{y_j}{||y_j||}f_j(y) \right|\right|\\
&=||P_{j-1}(x-y)||+\frac{\big|\big|x_j||y_j||f_j(x)-y_j||x_j||f_j(y) \big|\big|}{||x_j||\,||y_j||}\\
&\le \left(5+\frac{r_jA_{j-1}}{r_{j-1}}\right)||x-y||. \end{aligned}$$

If $\sum\limits_{i=1}^{j}\frac{||x_i||}{r_i}\le1$ and $\sum\limits_{i=1}^{j}\frac{||y_i||}{r_i}>1$, then from the definition of $f_j$ it follows that $||x_j||\le f_j(x)$ and $||y_j||-f_j(y)>0$ so
$$\begin{aligned}\frac{\big|\big|x_j||y_j||-y_jf_j(y)\big|\big|}{||y_j||}&\le\frac{\big|\big|x_j||y_j||-y_j||y_j||\,\big|\big|+\big|\big|y_j||y_j||-y_jf_j(y)\big|\big|}{||y_j||}\\
&=||x_j-y_j||+\big(||y_j||-f_j(y)\big)\\
&=||x_j-y_j||+\big(||y_j||-||x_j||\big)+\big(||x_j||-f_j(y)\big)\\
&\le2||x_j-y_j||+\big(f_j(x)-f_j(y)\big)\\
&\le\left(4+\frac{r_jA_{j-1}}{r_{j-1}}\right)||x-y||.
\end{aligned}$$
Hence, we deduce that
$$\begin{aligned} &||F_{j,j}(x)-F_{j,j}(y)||\\
&\le ||P_{j-1}(x-y)||+||(F_{j,j}-P_{j-1})(x)-(F_{j,j}-P_{j-1})(y)||\\
&=||P_{j-1}(x-y)||+\left|\left| x_j-\frac{y_j}{||y_j||}f_j(y) \right|\right|\\
&=||P_{j-1}(x-y)||+\frac{\big|\big|x_j||y_j||-y_jf_j(y) \big|\big|}{||y_j||}\\
&\le \left(5+\frac{r_jA_{j-1}}{r_{j-1}}\right)||x-y||. \end{aligned}$$

The case when  $\sum\limits_{i=1}^{j}\frac{||x_i||}{r_i}\le1$ and $\sum\limits_{i=1}^{j}\frac{||y_i||}{r_i}\le1$ is trivially true because $F_{j,j}$ acts as the identity, so we have proven \eqref{1ineqFDD}.\\

Finally, we prove \eqref{2ineqFDD} distinguishing between 3 different cases:\\

If $\sum\limits_{i=1}^{j-1}\frac{||x_i||}{r_i}\ge1$ then $F_{j,j}(x)-F_{j,j}(y)=P_{j-1}(x-y)$ so it is straightforward.\\

If $\sum\limits_{i=1}^{j-1}\frac{||x_i||}{r_i}<1<\sum\limits_{i=1}^{j}\frac{||x_i||}{r_i}$ then, as $1\le\sum\limits_{i=1}^{j-1}\frac{||y_i||}{r_i}$, it holds that
$$\begin{aligned}|f_j(x)|&=r_j\left(1-\sum\limits_{i=1}^{j-1}\frac{||x_i||}{r_i}\right)\\
&\le r_j\left(\sum\limits_{i=1}^{j-1}\frac{||y_i||-||x_i||}{r_i}\right)=|f_j(x)-f_j(y)|\le \frac{r_jA_{j-1}}{r_j}||x-y||.\end{aligned}$$
Hence,
$$\begin{aligned}||F_{j,j}(x)-F_{j,j}(y)||&=\left|\left| P_{j-1}(x-y)+\frac{x_j}{||x_j||}f_j(x) \right|\right|\\
&\le \left(1+\frac{r_jA_{j-1}}{r_{j-1}}\right)||x-y||.\end{aligned}$$

Finally, if $\sum\limits_{i=1}^{j}\frac{||x_i||}{r_i}\le1$, then
$$\begin{aligned}||x_j||&\le r_j\left( 1-\sum\limits_{i=1}^{j-1}\frac{||x_i||}{r_i} \right)\le r_j\left(\sum\limits_{i=1}^{j-1}\frac{||y_i||-||x_i||}{r_i}\right)=|f_j(x)-f_j(y)|\\
&\le \frac{r_jA_{j-1}}{r_{j-1}}||x-y||,\end{aligned}$$
and so,
$$\begin{aligned}||F_{j,j}(x)-F_{j,j}(y)||&=\left|\left| P_{j-1}(x-y)+x_j \right|\right|\\
&\le \left(1+\frac{r_jA_{j-1}}{r_{j-1}}\right)||x-y||.\end{aligned}$$
\end{proof}
\end{prop}

\begin{theorem}\label{theoFDDcompact}
There is a sequence $(q_n)_{n\in\N}\subset \R^+$ such that for every Banach space $X$ with a monotone FDD $(X_n)$, and every decreasing sequence $(r_n)\subset \R^+$ satisfying $\frac{r_n}{r_{n-1}}\le q_n$, the set
$$K=\overline{\co}\left(\bigcup\limits_{n\in\N}r_nB_{X_n}\right)$$
is a compact Lipschitz retract of $X$.
\begin{proof}
Consider some $\delta>0$ and $(\delta_n)_{n\in\N}\subset\R^+$ such that
$$\prod\limits_{n\in\N}(1+\delta_n)\le 1+\delta,$$
and consider a sequence $(a_n)_{n\in\N}\subset\R^+$ such that the sequence given by
$$(\alpha_n)_{n\in\N}:=\left(\sum\limits_{k=n+1}^\infty a_kA_k\right)_{n\in\N}$$ verifies that
$$\alpha_n\le\frac{\delta_n}{2}\;\;\;\;\;\;\forall n\in\N.$$
Now we set
$$q_n=\min\left\{a_n,\frac{\delta_n}{2A_{n-1}}\right\}.$$
Suppose that $(r_n)$ is a sequence as in the statement of the Theorem. Given any $n\in\N$ we define the following retraction
$$F_n=F_{n,1}\circ\cdots\circ F_{n,n}\circ P_n:X\rightarrow K_n.$$
For a given $x\in X$ and $n\in\N$ we set $\widetilde{m}=\max\Big\{k\in\{1,\dots,n+1\}\;:\;\sum\limits_{i=1}^{k-1}\frac{||x_i||}{r_i}\le1\Big\}$ so that it is possible to compute $F_n(x)$ as
$$F_n(x)=\begin{cases}P_{\widetilde{m}-1}(x)+\frac{x_{\widetilde{m}}}{||x_{\widetilde{m}}||}f_{\widetilde{m}}(x),\;\;\;&\text{if }\widetilde{m}\le n,\\
P_n(x)&\text{if }\widetilde{m}=n+1.\end{cases}$$
Now, if $x,y\in P_n(X)$, we claim that $||F_n(x)-F_n(y)||\le5(1+\delta)$. Indeed, let us consider
$$m=\max\left\{k\in\{1,\dots,n+1\}:\sum\limits_{i=1}^{k-1}\frac{||x_i||}{r_i}\le1,\sum\limits_{i=1}^{k-1}\frac{||y_i||}{r_i}\le1\right\}.$$

If $m=n+1$ then $F_n(x)-F_n(y)=x-y$.\\

If  $m=n$ then $F_{n}(x)=F_{n,n}(x)$ and $F_n(y)=F_{n,n}(y)$ and we use Proposition \ref{proplip2} to finish this case.\\

If $m=n-1$ then $F_n(x)=F_{n,n-1}(F_{n,n}(x))$ and $F_n(y)=F_{n,n-1}(F_{n,n}(y))$, and we know from the definition of $m$ that $x$ or $y$ verifies \eqref{2ineqFDD} of Proposition \ref{proplip2} for $j=n$. Hence, using Lemma \ref{lemmalip2} together with Proposition \ref{proplip2} we get that
$$\begin{aligned} ||F_n(x)-F_n(y)||&=||F_{n,n-1}(F_{n,n}(x))-F_{n,n-1}(F_{n,n}(y))||\\
&\le \frac{r_nA_n}{r_{n-1}}||F_{n,n}(x)-F_{n,n}(y)||\\
&+\big|\big|(F_{n-1,n-1}\circ P_{n-1})(F_{n,n}(x))-(F_{n-1,n-1}\circ P_{n-1})(F_{n,n}(y))\big|\big|\\
&\le\left(5+\frac{r_{n-1}A_{n-2}}{r_{n-2}}+\frac{r_nA_n}{r_{n-1}}\right)||F_{n,n}(x)-F_{n,n}(y)||\\
&\le\left(5+\frac{r_{n-1}A_{n-2}}{r_{n-2}}+\frac{r_nA_n}{r_{n-1}}\right)\left( 1+\frac{r_nA_{n-1}}{r_{n-1}} \right)||x-y||\\
&\le\left( 5+\frac{\delta_{n-1}}{2}+\alpha_{n-1} \right)\left( 1+\frac{\delta_n}{2} \right)||x-y||\le5(1+\delta)||x-y||.
\end{aligned}$$

Otherwise, if $m\le n-2$ then
$$F_n(x)=F_{n,m}\circ\cdots\circ F_{n,n}(x)\text{ and }F_n(y)=F_{n,m}\circ\cdots\circ F_{n,n}(y).$$
Also, from the definition of $m$ we get that, if $n\ge p\ge m+2$, then the point $F_{n,p}\circ\cdots\circ F_{n,n}(x)$ or the point $F_{n,p}\circ\cdots\circ F_{n,n}(y)$ satisfies \eqref{2ineqFDD} of Proposition \ref{proplip2} for $j= p-1$. Also the point $x$ or the point $y$ satisfies \eqref{2ineqFDD} for $j=n$. Having this in mind and using the same argument as in the previous step, we check that
$$\begin{aligned} &||F_n(x)-F_n(y)||=||F_{n,m}\circ\cdots\circ F_{n,n}(x)-F_{n,m}\circ\cdots\circ F_{n,n}(y)||\\
&\le\left( 5+\frac{r_mA_{m-1}}{r_{m-1}}+\sum\limits_{k=m+1}^{n}\frac{r_kA_k}{r_{k-1}} \right)||F_{n,m+1}\circ\cdots\circ F_{n,n}(x)-F_{n,m+1}\circ\cdots\circ F_{n,n}(y)||\\
&\le\left(5+\frac{r_mA_{m-1}}{r_{m-1}}+\sum\limits_{k=m+1}^{n}\frac{r_kA_k}{r_{k-1}}\right)\left( 1+ \frac{r_{m+1}A_m}{r_m} +\sum\limits_{k=m+2}^n\frac{r_kA_k}{r_{k-1}} \right)\\
&\cdot||F_{n,m+2}\circ\cdots\circ F_{n,n}(x)-F_{n,m+2}\circ\cdots\circ F_{n,n}(y)||\le\cdots\\
&\cdots\le \left(5+\frac{r_mA_{m-1}}{r_{m-1}}+\sum\limits_{k=m+1}^{n}\frac{r_kA_k}{r_{k-1}}\right)\left(\prod\limits_{j=m+1}^{n-1}\left( 1+ \frac{r_{j}A_{j-1}}{r_{j-1}} +\sum\limits_{k=j+1}^n\frac{r_kA_k}{r_{k-1}} \right)\right)\\
&\cdot||F_{n,n}(x)-F_{n,n}(y)||\\
&\le 5\left(\prod\limits_{j=m}^{n-1}\left( 1+ \frac{r_{j}A_{j-1}}{r_{j-1}} +\sum\limits_{k=j+1}^n\frac{r_kA_k}{r_{k-1}} \right)\right)\left(1+\frac{r_nA_{n-1}}{r_{n-1}}\right)||x-y||\\
&\le5\left(\prod\limits_{j=m}^{n-1}\left( 1+\frac{\delta_j}{2}+\alpha_j \right)\right)\left( 1+\frac{\delta_n}{2} \right)||x-y||\le 5\prod\limits_{j=m}^n(1+\delta_j)||x-y||\\
&\le5(1+\delta)||x-y||.
\end{aligned}$$

It is easy to see that $\forall x\in P_n(X)$, if $k>n$ then $F_k(x)=F_n(x)$ so we may define the following map
$$F:\bigcup\limits_{n\in\N}P_n(X)\rightarrow\co\left(\bigcup\limits_{n\in\N}r_nB_{X_n}\right),$$
$$F(x)=\lim\limits_{n\to\infty}F_n(x),$$
which is a $5(1+\delta)$-Lipschitz retraction. Considering now $R:X\rightarrow K$ as the extension of $F$ to the whole $X$, we are done.
\end{proof}
\end{theorem}

\begin{remark}\label{quotient}
We may actually choose $q_n=\frac{1}{n2^{n+1}}$. This arises from choosing $\delta_n=2^{-n+1}$ and $a_k=\frac{1}{k2^{k+1}}$, so that $\alpha_n=2^{-n}$.
\end{remark}

Notice that the restriction of the previous retraction $R$ to one of the blocks of the FDD is the radial projection. Hence, it is not possible to obtain an estimate for the Lipschitz norm of $R$ better than 2 for general FDD spaces.\\
Next, we are going to treat the special case when the blocks of the FDD are of dimension 1, that is when $X$ has a Schauder basis, which leads to a much
 better estimate on the Lipschitz norm of the retraction.

From now on, $dim X_n=1$ for every $n\in\N$ and $(e_n,e_n^*)_{n\in\N}$ is a monotone Schauder basis in $X$ such that $||e_n||=1$  and $\langle e_n \rangle=X_n$. From now on $r_n\in\R^+$ is a decreasing sequence, $(P_n)$ is the sequence of projections of the basis and we keep denoting $x_i=(P_i-P_{i-1})(x)$ for every $x\in X$. We are going to keep the previous definition for $F_{n,m}$ and $f_m$, as well as for $E_{n,m}$. The main difference between this case and the general one is stated in Proposition \ref{FDD1D}.

\begin{prop}\label{FDD1D}
$$F_{m,m}\;\text{ is }\;\left(1+\frac{r_mA_{m-1}}{r_{m-1}}\right)\text{-Lipschitz}\;\;\;\;\;\forall m\in \N.$$
\begin{proof}
We are going to prove it case by case. We have already proved in Proposition \ref{proplip2} the case when $x,y\in E_{m,m}=P_m(X)$ are such that $\sum\limits_{i=1}^{m-1}\frac{||y_i||}{r_i}>1$ or $\sum\limits_{i=1}^{m-1}\frac{||x_i||}{r_i}>1$ , so let us assume throughout all the proof that $\sum\limits_{i=1}^{m-1}\frac{||x_i||}{r_i}\le1$ and $\sum\limits_{i=1}^{m-1}\frac{||y_i||}{r_i}\le1$.

If $\sum\limits_{i=1}^{m}\frac{||x_i||}{r_i}\le 1$ and $\sum\limits_{i=1}^{m}\frac{||y_i||}{r_i}\le1$ then $F_{m,m}(x)-F_{m,m}(y)=x-y$ so it is straightforward that $||F_{m,m}(x)-F_{m,m}(y)||=||x-y||$.\\

If $\sum\limits_{i=1}^{m}\frac{||x_i||}{r_i}\le1$ and $\sum\limits_{i=1}^{m}\frac{||y_i||}{r_i}>1$, then we split this case into 2 different subcases. For these subcases we are going to set
$$t=\frac{e_m^*(x)-\frac{e_m^*(y)}{|e_m^*(y)|}f_m(y)}{e_m^*(x)-e_m^*(y)}$$
whenever $e_m^*(x)\neq e_m^*(y)$ and $t=\infty$ otherwise.\\

Subcase $t\notin[0,1]$.\\
In this subcase $|e_m^*(y)|=||y_m||>f_m(y)\ge0$ so then $e^*_m(x)\neq0$ because otherwise $t=\frac{f_m(y)}{|e_m^*(y)|}\in[0,1]$. We claim that in this subcase $\frac{e_m^*(x)}{|e_m^*(x)|}=\frac{e_m^*(y)}{|e_m^*(y)|}$. Indeed, if not, multiplying and dividing $t$ by $\frac{e^*_m(x)}{|e^*_m(x)|}=\frac{|e^*_m(x)|}{e^*_m(x)}=-\frac{|e^*_m(y)|}{e^*_m(y)}$, it is easy to see that
$$t=\frac{|e_m^*(x)|+f_m(y)}{|e_m^*(x)|+|e_m^*(y)|}\in[0,1],$$
which is not possible. This means that $\frac{e_m^*(x)}{|e_m^*(x)|}=\frac{e_m^*(y)}{|e_m^*(y)|}$ so now it can be easily seen that
$$t=\frac{|e_m^*(x)|-f_m(y)}{|e_m^*(x)|-|e_m^*(y)|}.$$
As $|e_m^*(x)|-f_m(y)\ge |e_m^*(x)|-|e_m^*(y)|$ we claim that $|e_m^*(x)|-f_m(y)\ge0$. In fact, we check case by case and obtain the following scheme
$$
\begin{cases}
\text{if }\;t=\infty\;&\Rightarrow\;|e_m^*(x)|-|e_m^*(y)|=0\;\Rightarrow\;|e_m^*(x)|-f_m(y)\ge0,\\
\text{if }\;t<0\;&\Rightarrow\;|e_m^*(x)|-|e_m^*(y)|<0\;\Rightarrow\;|e_m^*(x)|-f_m(y)>0,\\
\text{if }\;t>1\;&\Rightarrow\;|e_m^*(x)|-|e_m^*(y)|>0\;\Rightarrow\;|e_m^*(x)|-f_m(y)>0.
\end{cases}
$$
Then,
$$\left|e_m^*(x)-\frac{e_m^*(y)}{|e_m^*(y)|}f_m(y)\right|=|e_m^*(x)|-f_m(y)\le |f_m(x)-f_m(y)|\le \frac{r_mA_{m-1}}{r_{m-1}}||x-y||,$$
and we have
$$\begin{aligned}||F_{m,m}(x)-F_{m,m}(y)||&=\left|\left|P_{m-1}(x-y)+\left(x_m-\frac{y_m}{||y_m||}f_m(y)\right)\right|\right|\\
&\le||x-y||+\bigg|e_m^*(x)-\frac{e_m^*(y)}{|e_m^*(y)|}f_m(y) \bigg|\\
&\le||x-y||+\frac{r_mA_{m-1}}{r_{m-1}}||x-y||.\end{aligned}$$

The subcase $t\in[0,1]$  is simpler. Due to the convexity of the norm and the fact that $||P_{m-1}||=1$ we have
$$||F_{m,m}(x)-F_{m,m}(y)||=||P_{m-1}(x-y)+t(x_m-y_m)||\le||x-y||,$$
so we are finally done with both subcases. Now, if $\sum\limits_{i=1}^{m}\frac{||x_i||}{r_i}>1$ and $\sum\limits_{i=1}^{m}\frac{||y_i||}{r_i}>1$ then $||x_m||>0$ and $||y_m||>0$ and we are also spliting this case into 2 subcases for technical reasons: \\

Subcase $\frac{x_m}{||x_m||}=\frac{y_m}{||y_m||}$. Here,
$$\left|\left|\frac{x_m}{||x_m||}f_m(x)-\frac{y_m}{||y_m||}f_m(y)\right|\right|=|f_m(x)-f_m(y)|\le\frac{r_mA_{m-1}}{r_{m-1}}||x-y||,$$
so then
$$\begin{aligned}||F_{m,m}(x)-F_{m,m}(y)||&=\left|\left|P_{m-1}(x-y)+\left(\frac{x_m}{||x_m||}f_m(x)-\frac{y_m}{||y_m||}f_m(y)\right)\right|\right|\\
&\le \left(1+\frac{r_mA_{m-1}}{r_{m-1}}\right)||x-y||.\end{aligned}$$

Finally, for the subcase when $\frac{x_m}{||x_m||}=-\frac{y_m}{||y_m||}$ we consider
$$t=\frac{f_m(x)+f_m(y)}{||x_m||+||y_m||}\in[0,1],$$
and again by convexity,
$$||F_{m,m}(x)-F_{m,m}(y)||=\left|\left|P_{m-1}(x-y)+t(x_m-y_m)\right|\right|\le||x-y||.$$
\end{proof}
\end{prop}

\begin{theorem}\label{SB}
For every $\delta>0$ there exists a sequence $(q_n)\subset \R^+$ such that for every Banach space $X$ with a monotone Schauder basis $(e_n)$ and every decreasing sequence $(r_n)\subset\R^+$ satisfying $\frac{r_n}{r_{n-1}}\le q_n$, the set
$$K=\overline{\co}\left(\bigcup\limits_{k\in\N}r_kB_{\langle e_k\rangle}\right)$$
is a $(1+\delta)$-Lipschitz retract of $X$.
\begin{proof}
Consider $(\delta_n)_{n\in\N}\subset\R^+$ such that
$$\prod\limits_{n\in\N}(1+\delta_n)\le 1+\delta,$$
and consider a sequence $(a_n)_{n\in\N}\subset\R^+$ such that the sequence given by
$$(\alpha_n)_{n\in\N}:=\left(\sum\limits_{k=n+1}^\infty a_kA_k\right)_{n\in\N}$$ verifies that
$$\alpha_n\le\frac{\delta_n}{2}\;\;\;\;\;\;\forall n\in\N.$$
Now we set
$$q_n=\min\left\{a_n,\frac{\delta_n}{2A_{n-1}}\right\}.$$
Suppose that $(r_n)$ is as in the statement of the Theorem. We fix $n>1$, so that making use of Lemma \ref{lemmalip2} and Proposition \ref{FDD1D}, the following holds
$$\begin{aligned}
&\text{If }\;m=1\;&\Rightarrow\; &||F_{n,1}||_{Lip}\le1+\sum\limits_{k=2}^n\frac{r_kA_k}{r_{k-1}}.\\
&\text{If }\;m\in\{2,\dots,n-1\}\;&\Rightarrow\; &||F_{n,m}||_{Lip}\le1+ \frac{r_mA_{m-1}}{r_{m-1}} +\sum\limits_{k=m+1}^n\frac{r_kA_k}{r_{k-1}}.\\
&\text{If }\;m=n\;&\Rightarrow\; &||F_{n,n}||_{Lip}\le1+ \frac{r_nA_{n-1}}{r_{n-1}}.
\end{aligned}$$
Let us consider now the composed retraction
$$F_n=F_{n,1}\circ\cdots\circ F_{n,n}\circ P_{n}:X\rightarrow K_n.$$
As in Theorem \ref{theoFDDcompact} it is enough to show that $||F_n||_{Lip}\le1+\delta$:
$$\begin{aligned}||F_n||&_{Lip}\le\prod\limits_{m=1}^n||F_{n,m}||_{Lip}\\
\le&\left( 1+\sum\limits_{k=2}^n\frac{r_kA_k}{r_{k-1}} \right)\left( \prod_{m=2}^{n-1} \left(1+ \frac{r_mA_{m-1}}{r_{m-1}} +\sum\limits_{k=m+1}^n\frac{r_kA_k}{r_{k-1}}\right) \right)\\
&\cdot\left( 1+ \frac{r_nA_{n-1}}{r_{n-1}} \right)\\
\le& (1+\alpha_1)\left(\prod\limits_{m=2}^{n-1}\left( 1+\frac{\delta_m}{2}+\alpha_m \right)\right)\left( 1+\frac{\delta_n}{2} \right)\le \prod\limits_{m=1}^n(1+\delta_m)\le1+\delta.\end{aligned}$$
\end{proof}
\end{theorem}


\section{$\pi$-property and compact Lipschitz retractions}

We pass to the results concerning the necessary conditions on the Banach space $X$ so that $X$ admits
a GCCR $K\subset X$. Our methods require a certain quantitative "smallness" condition to be satisfied for $K$.
Under such assumption we show that $X$ must have the $\pi$-property. In fact, our argument makes no use of the convexity of $K$. The crucial
condition is smallness. Our proof uses three main ingredients. The unpublished Milman lemma (communicated to us, with proof, by Bill Johnson)
concerning the projection constant of a finite dimensional subspace of a Banach space,
the finite dimensional  "Lipschitzization"
of coarse Lipschitz maps due to Bourgain (and streamlined by Begun),
 and the averaging of derivatives for finite dimensional
Lipschitz maps. We start with a well-known fact.

 Given $r\in\R^+$, we are going to denote $[r]=\max\{n\in\N\cup\{0\}\;:\;n\le r\}$.

\begin{lemma}\label{MilSecht}
For every $n\in \N$ and $\ep\in(0,1)$, if $(E, ||\cdot||)$ is a Banach space of dimension $n$, then there exists a renorming $|\cdot|$ of $E$ such that $(E,|\cdot|)$ embeds isometrically in $\ell_\infty^N$ where $N=\big[(1+2/\ep)^n\big]$ and
$$|x|\le||x||\le\frac{|x|}{1-\ep}.$$
\begin{proof}
By Lemma 2.6 of \cite{MS86}, we know that there exists an $\ep$-net in $S_{X^*}$ consisting of $N$ points, namely $\{x_1^*,\dots,x_N^*\}$. Just consider the norm $|x|=\max\limits_{i\in\{1,\dots,N\}}x_i^*(x)$.
\end{proof}
\end{lemma}

If $X$ is a Banach space and $E\subset X$ is a subspace, then the projection constant of $E$ in $X$ is defined as
$$\lambda(E,X)=\inf\big\{||P||\;:\;P:X\rightarrow E\;,\;\restr{P}{E}=Id_E\big\}.$$

\begin{lemma}[Vitali Milman-unpublished]\label{lemmaJohnson}
Let $X$ be a Banach space. For every $\ep\in(0,1)$ and a subspace $E\subset X$  of dimension $dim(E)=n$, there is another subspace $G_E\subset X$ containing $E$ such that $dim(G_E)\le\big( 1+\frac{2}{\ep} \big)^n$ and
$$\lambda(E,X)\le \frac{2}{1-\ep}\lambda(E,G_E).$$
\begin{proof}(communicated to us by Bill Johnson)
We follow the trace duality arguments set up  in  \cite{JKM79}.
Pick $\ep\in(0,1)$. For a given Banach space $Y$ with a finite dimensional subspace $E\subset Y$ we define a pair of norms on the space of all linear operators $\mathcal{L}(E)$
$$||T||_Y=\inf\big\{ ||\widetilde{T}||\;:\; \widetilde{T}\in\mathcal{L}(Y,E)\;,\;\restr{\widetilde{T}}{E}=T \big\},$$
$$||T||_{\Lambda Y}=||i_ET||_{\Lambda},$$
where $i_E:E\rightarrow Y$ is the inclusion map and $||\cdot||_{\Lambda}$ refers to the nuclear norm in $\mathcal{L}(E,Y)$, that is,
$$||T||_{\Lambda Y}=\inf\bigg\{ \sum\limits_{i=1}^n||x_i^*||\cdot||y_i||\;:\;n\in\N,T=\sum\limits_{i=1}^nx_i^*\otimes y_i, x_i^*\in E^*,y_i\in Y \bigg\} .$$
We know from \cite{JKM79} pg. 377 that both norms are in trace duality. More precisely, we have a dual pairing $\langle\mathcal{L}(E), \mathcal{L}(E)\rangle$ 
given by $\langle T,S \rangle=tr(ST)$ for every $T,S\in\mathcal{L}(E)$, such that
$$(\mathcal{L}(E),||\cdot||_Y)^*=(\mathcal{L}(E),||\cdot||_{\Lambda Y}).$$
Thanks to this interpretation, we can compute
$$\begin{aligned}\lambda(E,Y)&=||Id_E||_Y=\sup\limits_{||T||_{\Lambda Y}=1}tr(TId_E)=\sup\limits_{T\in\mathcal{L}(E)}tr\bigg(\frac{T}{||T||_{\Lambda Y}}\bigg)\\&=\sup\limits_{T\in\mathcal{L}(E)}\frac{1}{\bigg|\bigg| \frac{T}{tr(T)} \bigg|\bigg|_{\Lambda Y}}=\sup\limits_{tr(T)=1}\frac{1}{||T||_{\Lambda Y}}=\frac{1}{\inf\limits_{tr(T)=1}||T||_{\Lambda Y}}.\end{aligned}$$

Returning to the situation of our theorem, $E$ is now a subspace of $X$. Let us first take $\mu\in(0,1/2)$. Now we take, for $\delta=\frac{(1-\ep)(1-\mu)-(1-\ep)1/2}{\lambda(E,X)}>0$, a trace one operator   $S\in\mathcal{L}(E)$ such that
$$\inf\limits_{tr(T)=1}||T||_{\Lambda X}\ge ||S||_{\Lambda X}-\delta.$$
We also take the norm $|\cdot|$ given by Lemma \ref{MilSecht} so that $(E,|\cdot|)$ is isometrically a subspace of $\ell_\infty^{\varphi(n)}$ where $\varphi(n)=\Big[\big( 1+\frac{2}{\ep} \big)^n\Big]$ and, denoting $|\cdot|_{\Lambda Y}$ the nuclear norm taking $(E,|\cdot|)$ as the domain of the operators instead of $(E,||\cdot||)$, we have for every superspace $Y\supset E$ that
$$ ||S||_{\Lambda Y}\le|S|_{\Lambda Y}\le \frac{||S||_{\Lambda Y}}{1-\ep}. $$
It is well-known (Proposition 47.6 in \cite{Tre06}) that $i_ES$ admits an extension $\widetilde{S}:\ell_\infty^{\varphi(n)}\rightarrow X$ almost preserving the nuclear norm, that is $|S|_{\Lambda X}\ge(1-\mu)||\widetilde{S}||_\Lambda$. By Proposition 8.7 from \cite{Tom89} we know that there exist $x_1,\dots,x_{\varphi(n)}\in X$ such that $\widetilde{S}=\sum\limits_{i=1}^{\varphi(n)}e_i^*\otimes x_i$ and
$$||\widetilde{S}||_{\Lambda}=\sum\limits_{i=1}^{\varphi(n)}||x_i||,$$
where $e_i^*\in\big(\ell_\infty^{\varphi(n)}\big)^*$ are the coordinate functionals.
Just considering $G_E=[x_i]_{i=1}^{\varphi(n)}$, we can see that
$$\begin{aligned}||S||_{\Lambda X}&\ge (1-\ep)|S|_{\Lambda X}\ge(1-\ep)(1-\mu)||\widetilde{S}||_{\Lambda}\\&\ge(1-\ep)(1-\mu)|S|_{\Lambda G_E}\ge||S||_{\Lambda G_E}(1-\ep)(1-\mu).\end{aligned}$$
Finally, taking into acount that $\lambda(E,G_E)\le\lambda(E,X)$,   we finish the proof because
$$\begin{aligned}\lambda(E,X)&\le\frac{1}{||S||_{\Lambda X}-\delta}\le \frac{1}{(1-\ep)(1-\mu)||S||_{\Lambda G_E}-\delta}
\\&\le\frac{1}{(1-\ep)(1-\mu)\inf\limits_{tr(T)=1}||T||_{\Lambda G_E}-\delta}\\&=\lambda(E,G_E)\frac{1}{(1-\ep)(1-\mu)-\delta\lambda(E,G_E)}\le\frac{2}{1-\ep}\lambda(E,G_E).\end{aligned}$$
\end{proof}
\end{lemma}

\begin{definition}
Given a separable Banach space $X$, we will say that $\beta=(e_n)\subset X$ is a fundamental sequence if $[e_n]=X$ and we will denote $E_n^{\beta}=[e_i]_{i=1}^n$.
\end{definition}

\begin{definition}
Let $X$ be a separable Banach space, $\beta=(e_n)$ a fundamental sequence and $K\subset X$ a bounded subset. We will define the following concepts:
\begin{itemize}
\item The sequence of inner radii $(r_n^\beta)$ given by
$$r_n^\beta=\sup\big\{r\ge0\;:\;B_{E_n^\beta}(x,r)\subset K\cap E_n^{\beta}\;,\;x\in X\big\}\;\;\;\;\forall n\in\N.$$
\item The sequence of heights $(h_n^\beta)$ given by
$$h_n^\beta=\sup\big\{d(x,E_n^\beta)\;:\;x\in K\big\}\;\;\;\;\forall n\in\N.$$
\end{itemize}
\end{definition}

\begin{definition}\label{smallcompact}
We say that a bounded subset $K$ of a separable Banach space $X$ is small if there exist an $\ep\in(0,1)$, a fundamental sequence $\beta=(e_n)$ in $X$ and a strictly increasing $\sigma:\N\rightarrow \N$ such that
$$0<\frac{h_{\sigma(n)}^\beta}{r_{\sigma(n)}^\beta}\le\frac{1}{2\sigma(n)^2\Big(\big(1+\frac{2}{\ep}\big)^{\sigma(n)}+2\Big)}\;\;\;\;\forall n\in\N.$$
\end{definition}

Note that such sets are necessarily compact and generate $X$.\\

If $\beta=(e_n)$ is a monotone Schauder basis and $X_i=\langle e_i\rangle$, then we know from Theorem \ref{theoFDDcompact} and Remark \ref{quotient} that for every sequence $(r_n)\subset\R^+$ such that
$$\frac{r_n}{r_{n-1}}\le q_n=\frac{1}{n2^{n+1}},$$
the compact $K=\co\left(\bigcup\limits_{k\in\N}r_kB_{X_k}\right)$ is a GCCR. In this case it is easily seen that there is a $C>0$ independent of $n\in\N$ and $X$ with
$$\frac{h_{n}^\beta}{r_{n}^\beta}\le C\frac{q_n}{n}=\frac{C}{n^22^{n+1}},$$
where the right hand side of the inequality is a very similar sequence to the one given in the definition of smallness.

More generally, it is easy to check using Remark \ref{quotient} that the following result holds.

\begin{prop}
If a separable Banach space  $X$ has an FDD then $X$ admits a small GCCR.
\end{prop}

We now pass to the promised opposite implication. Note that the convexity assumption on the generating compact $K$ is not needed.

\begin{theorem}\label{theosmall}
Let $X$ be a separable Banach space. If there exists a Lipschitz retraction from $X$ onto a small compact subset, then $X$ has the $\pi$-property.
\begin{proof}
Assume that there is a Lipschitz retraction from $X$ onto a small compact $K$. Take $\ep\in(0,1)$, $\beta=(e_n)$ a fundamental sequence of $X$ and $\sigma:\N\rightarrow\N$ strictly increasing for which the inequality of Definition \ref{smallcompact} holds true, and let $\varphi(n)=\Big[\big(1+\frac{2}{\ep}\big)^{\sigma(n)}\Big]$, $E_n=E_{\sigma(n)}^\beta$, $h_n=h_{\sigma(n)}^\beta$ and $r_n=r_{\sigma(n)}^\beta$. Lemma \ref{lemmaJohnson} guarantees that for every $n\in\N$ there is a finite dimensional subspace $G_n\subset X$ of dimension $dim(G_n)=\varphi(n)$ such that for every projection $P: G_n\rightarrow E_n$, the inequality $||P||\ge (1-\ep)\frac{\lambda(E_n,X)}{2}$ holds. Assume that $R:X\rightarrow K$ is the Lipschitz retraction, then taking $C_n:K\rightarrow E_n$ a nearest point map (it may not be unique), we define $\widetilde{R}_n=\restr{(C_n\circ R)}{G_{n}}:G_{n}\rightarrow E_n$ for every $n\in\N$. Now,
$$||\widetilde{R}_n(x)-\widetilde{R}_n(y)||\le ||R||\bigg( ||x-y||+\frac{2h_{n}}{||R||} \bigg)\;\;\;\;\forall x,y\in G_{n},$$
so by the Proposition of \cite{B99}, for every $\tau>0$, there is a Lipschitz mapping 
$$R_{n,\tau}:G_{n}\rightarrow E_n$$
such that
$$||R_{n,\tau}||_{Lip}\le||R||\bigg( 1+\frac{\varphi(n)h_{n}}{||R||\tau} \bigg),$$
$$||R_{n,\tau}(x)-\widetilde{R}_n(x)||\le||R||\bigg( \tau+\frac{2h_{n}}{||R||} \bigg)\;\;\;\;\forall x\in G_{n}.$$

For the rest of the argument we fix  $x_n\in K_n:=K\cap E_n$ such that $B_{E_n}(x_n,r_n)\subset K_n$ (it exists by the definition of the inner radius and the compactness).
Now we choose $\tau_n=\frac{\varphi(n)h_{n}}{||R||}$, and define $R_n:G_n\rightarrow E_n$ by $R_n(x)=R_{n,\tau_n}(x+x_n)-x_n$. 
If $x+x_n\in K_n$ then $\widetilde{R}_n(x+x_n)=x+x_n$ acts as an identity.
Hence we have that for every $x\in K_n+\{-x_n\}$
$$||R_n(x)-x||=||R_{n,\tau_n}(x+x_n)-\widetilde{R}_n(x+x_n)||\le h_{n}(\varphi(n)+2)=:\rho_n.$$
Now, let $(a_i,a_i^*)_{i=1}^{\varphi(n)}$ be a normalized linear basis for $G_n$ with projections $(S_i)_{i=1}^{\varphi(n)}$ such that $(a_i,a_i^*)_{i=1}^{\sigma(n)}$ is an Auerbach basis for $E_n$. Then, 
$$B_n=r_n\co\big(\{\pm a_i\;,\;i=1,\dots,\sigma(n)\}\big)\subset K_n+\{-x_n\}.$$
Fix a sequence $(\delta_k)$ of positive numbers converging to zero. Now we define for every $k\in\N$ the compact
$$B_{n,k}=B_n+\delta_k\sum\limits_{i=\sigma(n)+1}^{\varphi(n)}[-a_i,a_i]\subset G_n.$$
The geometrical shape of this set can be described as a $r_n$-multiple of the unit ball of $\ell_1^{\sigma(n)}$, a sort of a base of a hypercylinder, located in $E_n$ times a hypercube of side length $2\delta_k$ sticking into the remaining dimensions of $G_n$.
Letting $\delta_k$ go to zero of course means that this set gets squashed down to its base in $E_n$.
For any $i\in\{1,\dots,\sigma(n)\}$ we denote $(B_{n,k})_i=(Id-a_i^*\otimes a_i)(B_{n,k})$. This set is just a one-codimensional section (or a projection of rank $\varphi(n)-1$) of $B_{n,k}$ which reduces the base by one coordinate.

In order to recover the shape $B_{n,k}$ from its section $(B_{n,k})_i$ we pass from any point  $x^i\in (B_{n,k})_i$ to boundary point of $B_{n,k}$ which got projected onto it. The newly acquired coordinate vector will then be denoted by $x_i(x^i)$ and 
given by the formula

$$x_i(x^i)=\bigg( r_n-\sum\limits_{\substack{j=1\\j\neq i}}^{\sigma(n)}a_j^*(x^i) \bigg)a_i,\;\; x^i\in (B_{n,k})_i.$$

For convenience in our computations, we also introduce the quantity
$$z_i(x^i)=R_n(x^i+x_i(x^i))-R_n(x^i-x_i(x^i))-2x_i(x^i).$$

As $S_{\sigma(n)}$ is a linear projection, we know that
$-S_{\sigma(n)}\big(x^i+x_i(x^i)\big)+S_{\sigma(n)}\big(x^i-x_i(x^i)\big)=-2x_i(x^i)$. Using this and a triangle inequality with four terms 
we have that
$$\begin{aligned}||z_i(x^i)||\le&\big|\big|R_n\big(x^i+x_i(x^i)\big)-R_n\big(S_{\sigma(n)}\big(x^i+x_i(x^i)\big)\big)\big|\big|\\&+\big|\big|R_n\big(S_{\sigma(n)}\big(x^i+x_i(x^i)\big)\big)-S_{\sigma(n)}\big(x^i+x_i(x^i)\big)\big|\big|\\&+\big|\big|R_n\big(S_{\sigma(n)}\big(x^i-x_i(x^i)\big)\big)-R_n\big(x^i-x_i(x^i)\big)\big|\big|\\&+\big|\big|S_{\sigma(n)}\big(x^i-x_i(x^i)\big)-R_n\big(S_{\sigma(n)}\big(x^i-x_i(x^i)\big)\big)\big|\big|\\\le&2||R_n||\varphi(n)\delta_k+2\rho_n.\end{aligned}$$
For a set $J\subset \{1,\dots,\varphi(n)\}$ with $\# J=m\ge1$, we define the measure in $[a_i]_{i\in J}$ as 
$$\lambda^m_J(A)=\lambda_m\bigg(\prod\limits_{i\in J}\Big( a_i^*\big(A\big)\Big)\bigg)\;\;\;\;\forall A\in \mathcal{M}^m_J, $$
where $\lambda_m$ is the Lebesgue measure in $\R^m$ and $\mathcal{M}^m_J=\bigg\{A\subset [a_i]_{i\in J}\;:\;\prod\limits_{i\in J}\Big( a_i^*\big(A\big)\Big) \text{ is Lebesgue measurable subset of }\R^m\bigg\}$. If $J=\{1,\dots,\varphi(n)\}$ we denote $\lambda^{\varphi(n)}=\lambda^{\varphi(n)}_J$ and for every $i\in\{1,\dots,\varphi(n)\}$, if $J=\{1,\dots,\varphi(n)\}\setminus\{i\}$, we denote $\lambda^{\varphi(n)-1}_i=\lambda^{\varphi(n)-1}_{J}$.
Then, we are ready to define  the linear operators $P_{n,k}:G_n\rightarrow E_n$, for every $k\in\N$, as
$$P_{n,k}(v)=\frac{1}{\lambda^{\varphi(n)}(B_{n,k})}\int_{B_{n,k}}dR_n(x)[v]d\lambda^{\varphi(n)}(x).$$

In \cite{Bra+14} pg. 47
the volumes of $\ell_p^n$ balls $B_p^n$ have been computed as $|B_p^n|=\frac{2^n\Gamma(\frac1p+1)^n}{\Gamma(\frac np+1)}$. Using this result for $p=1$ (for the base part of our set $B_{n,k}$) and the standard
properties of Lebesgue measure we obtain the following values for our sets for arbitrary $i\in\{1,\dots,\sigma(n)\}$
$$\lambda^{\varphi(n)-1}_i\big((B_{n,k})_i\big)=\frac{2^{\varphi(n)-1}r_n^{\sigma(n)-1}\delta_k^{\varphi(n)-\sigma(n)}}{(\sigma(n)-1)!},$$
$$\lambda^{\varphi(n)}\big(B_{n,k}\big)=\frac{2^{\varphi(n)}r_n^{\sigma(n)}\delta_k^{\varphi(n)-\sigma(n)}}{\sigma(n)!},$$
so the quotient is
$$\frac{\lambda^{\varphi(n)-1}_i((B_{n,k})_i)}{\lambda^{\varphi(n)}(B_{n,k})}=\frac{\sigma(n)}{2r_n}.$$
Note that the expression
$$R_n(x^i+x_i(x^i))-R_n(x^i-x_i(x^i))=z_i(x^i)+2x_i(x^i)$$
represents the difference of the values of the operator $R_n$ between the endpoints of a segment cutting through $B_{n,k}$, which passes through the point $x^i$ with direction $a_i$.
As $B_{n,k}=\{u+w\in G_n\;:\;u\in(B_{n,k})_i\;,\;w\in[-x_i(x^i),x_i(x^i)]\}$, thanks to Fubini's Theorem and the Fundamental Theorem of Calculus applied to the $i$-th coordinate, we can compute for each $i\in\{1,\dots,\sigma(n)\}$
$$\begin{aligned}\big|\big|a_i-P_{n,k}(a_i)\big|\big| 
&=\bigg|\bigg|a_i- \frac{1}{\lambda^{\varphi(n)}(B_{n,k})}\int_{B_{n,k}}dR_n(x)[v]d\lambda^{\varphi(n)}(x)\bigg|\bigg|\\ 
&=\bigg|\bigg|a_i- \frac{1}{\lambda^{\varphi(n)}(B_{n,k})}\int_{(B_{n,k})_i}z_i(x^i)+2x_i(x^i)d\lambda^{\varphi(n)-1}_i(x^i) \bigg|\bigg|\\&=\bigg|\bigg|\frac{1}{\lambda^{\varphi(n)}(B_{n,k})} \int_{(B_{n,k})_i}z_i(x^i)d\lambda^{\varphi(n)-1}_i(x^i) \bigg|\bigg|\\
&\le\frac{\lambda^{\varphi(n)-1}_i((B_{n,k})_i)}{\lambda^{\varphi(n)}(B_{n,k})}\big(2||R_n||\varphi(n)\delta_k+2\rho_n\big)\\
&=\frac{\sigma(n)||R_n||\varphi(n)\delta_k+\sigma(n)\rho_n}{r_n}.
\end{aligned}$$
We may assume that $P_{n,k}$ pointwise converge in $k$ and define $P_n(x)=\lim\limits_{k\to\infty}P_{n,k}(x)$ for every $x\in G_n$, which is a linear operator from $G_n$ to $E_n$ satisfying that
$$||P_n(a_i)-a_i||\le\frac{\sigma(n)h_{n}(\varphi(n)+2)}{r_n}\;\;\;\;\forall i\in\{1,\dots,\sigma(n)\}.$$
Using the fact that $K$ is small and $||a_i^*||=1$, $i\in\{1,\dots,\sigma(n)\}$, we obtain that for every $x\in E_n$
$$||P_n(x)-x||=\bigg|\bigg| \sum\limits_{i=1}^{\sigma(n)}a_i^*(x)(P_n(a_i)-a_i) \bigg|\bigg|\le\frac{h_{n}{\sigma(n)}^2(\varphi(n)+2)}{r_n}||x||\le\frac{1}{2}||x||.$$
Finally, we construct the projection $\widetilde{P}_n=\Big(\restr{P_n}{E_n}\Big)^{-1}\circ P_n:G_n\rightarrow E_n$ of norm
$$||\widetilde{P}_n||\le\Big|\Big| \Big(\restr{P_n}{E_n}\Big)^{-1} \Big|\Big|\cdot||P_n||\le2\cdot 2||R|| =4||R||\;\;\;\;\;\forall n\in\N.$$
This implies that $X$ has the $\pi$-property since
$$\lambda(E_n,X)\le\frac{2}{1-\ep}\lambda(E_n,G_n)\le\frac{2}{1-\ep}||\widetilde{P}_n||\le\frac{8||R||}{1-\ep}\;\;\;\;\forall n\in\N.$$
\end{proof}
\end{theorem}

Recall that thanks to Theorem 1.3 in \cite{Joh2}, dual Banach spaces have an FDD if and only if they enjoy the $\pi$-property. Hence we have the next characterization.
\begin{corollary}
Let $X$ be a separable dual  Banach space. Then $X$ has an FDD if and only if $X$ admits a small GCCR, if and only if $X$ admits a small subset which is a Lipschitz retract of $X$.
\end{corollary}

In Section \ref{secretr} we have found a sequence $(q_n)\subset \R^+$ such that for every sequence $r=(r_n)\subset \R^+$ satisfying that $\frac{r_n}{r_{n-1}}\le q_n$ there is a $\lambda$-Lipschitz retraction $R(r):X\rightarrow K(r)$, where $K(r)=\co\Big(\bigcup\limits_{k\in\N}r_kB_{X_k}\Big)$ for some FDD $(X_n)_{n\in\N}$.  Let $r=(r_n)$ be such a sequence and denote for every $k,m\in\N$ the sequence $r^{k,m}=(r_1,\dots,r_m,r_{m+1}/k,\dots,r_n/k,\dots)$. Taking subsequences, we may assume that for every $x\in X$ and every $m\in\N$ there exists $R_m(x)=\lim\limits_{k\to\infty}R(r^{k,m})(x)$ which define retractions onto increasing finite dimensional compacts. This leads to the $\pi$-property for Lipschitz retractions.

\begin{definition}
Let $X$ be a separable Banach space and $\lambda>0$. We say that $X$ has the Lipschitz $\pi_\lambda$-property if there exists an increasing sequence of finite dimensional convex subsets $(C_n)$ of $X$ such that $X=\overline{\bigcup\limits_{n\in\N}span (C_n)}$ and there exists a $\lambda$-Lipschitz retraction $R_n:X\rightarrow C_n$ for every $n \in\N$.
\end{definition}

Analogously to the result of Godefroy and Kalton on the Lipschitz bounded approximation property (Theorem \ref{GODKAL}), we are going to prove that this new property is nothing else but the well-known $\pi$-property. This result is a direct consequence of the next Theorem, which is mainly based on an
ultraproduct technique  similar to the result of Lindenstrauss in \cite{Lin64}, see Corollary 7.3 of \cite{BL2000}.

\begin{theorem}\label{complemented}
Let $X$ be a Banach space, $\lambda_1,\lambda_2>0$ and $Y\subset X$ a subspace $\lambda_1$-complemented in its bidual. If there is a $\lambda_2$-Lipschitz retraction from $X$ onto a convex subset $K$ containing 0 such that $\overline{span}(K)=Y$ then $Y$ is $\lambda_1\lambda_2$-complemented in $X$.
\begin{proof}
Let $R:X\rightarrow K\subset Y$ be such a retraction.  Then for every $n\in\N$ we define the $\lambda_2$-Lipschitz retraction $R_n:X\rightarrow nK$ given by $R_n(x)=nR(x/n)$, for every $x\in X$. 
As in the proof of Theorem 7.2 of \cite{BL2000} we let $\mathcal{U}$ to be a free ultrafilter on $\N$, and we put
$$S(x)=\lim\limits_{\mathcal{U}}R_{n}(x)\;\;\;\;\forall x\in X.$$
It is standard to check that $S: X\to Y^{**}$ is a $\lambda_2$-Lipschitz mapping, which is identity on $Y$.
 Now, if $L:Y^{**}\rightarrow Y$ is a bounded linear projection, we just define $\widetilde{R}:X\rightarrow Y$ by
$$\widetilde{R}(x)=L\circ S(x)\;\;\;\;\forall x\in X.$$
Finally $\widetilde{R}$ is a $\lambda_1\lambda_2$-Lipschitz retraction from $X$ onto $Y$ so by Corollary 7.3 of \cite{BL2000} we are done.
\end{proof}
\end{theorem}

\begin{corollary}
Let $X$ be a separable Banach space and $\lambda>0$. Then $X$ has the Lipschitz $\pi_\lambda$-property if and only if  it has the $\pi_\lambda$-property.
\begin{proof}
It is straightforward from Proposition \ref{conmutingpi} and Theorem \ref{complemented}.
\end{proof}
\end{corollary}


\section{Compacts without Lipschitz retractions}

We proceed by constructing an example of a small convex and compact $K$  in $C[0,1]$, which is contained in a small GCCR, contains a small GCCR, and yet
there is no Lipschitz retraction onto $K$.
The idea behind the construction can be described as follows.  The GCCR constructed at the beginning of our note are well "aligned" with
the FDD on $X$, and in the proof that the smallness condition of GCCR implies the $\pi$-property the projections are also aligned
with the structure of the compact. So our strategy is to employ badly complemented finite dimensional subspaces (in fact, Hilbert spaces) of
$C[0,1]$ as the sections of the sought compact $K$. In order to glue the decreasing sequence of these pieces together, we
use the $\mathcal{L}_\infty$-FDD in $C[0,1]$. We start with a standard argument.

\begin{lemma}\label{fddjohnson}

For any $\ep>0$ there is an FDD in $C[0,1]$ with associated projections $P_n$, and a sequence $(a_n)\subset C[0,1]$, $(1+\ep)$-equivalent to the $\ell_2$
basis such that
\begin{itemize}
\item $\sup\limits_{n\in \N}d(P_n(C[0,1]),\ell_\infty^{d(n)})<\infty$ where $d(n)=dim(P_n(C[0,1]))$.
\item $a_n\in (P_n-P_{n-1})(C[0,1])$ for every $n\in\N$.
\end{itemize}
\begin{proof}
By Remark 5.2 of \cite{Joh2} there is an FDD $(X_n)$ in $C[0,1]$ such that 
\begin{equation}\label{dd}
\sup\limits_{n\in\N}d(\tilde Q_n(C[0,1]),\ell_\infty^{d(n)})=d<\infty,
\end{equation}
where the $\tilde Q_n$ are the natural projections of the FDD given by $\tilde Q_n(\sum_{i=1}^\infty x_i)=\sum_{i=1}^n x_i$.
Take an $\ell_2$-basis $(b_n)_{n\in\N}\subset C[0,1]$, which is certainly a $w^*$-null normalized sequence. For a given $\ep>0$ apply the standard 
Bessaga-Pelczynski blocking principle (e.g. \cite{Fab1} p. 194) to obtain a subsequence $(b_{k(n)})$ of the original $\ell_2$ basis,
and its perturbed companion $(a_n)$, $\sum_{n=1}^\infty||a_n-b_{k(n)}||<\ep$, such that $a_n\in(\tilde Q_{\eta(n)}-\tilde Q_{\eta(n-1)})(C[0,1])$,
for some increasing sequence of indices $(\eta(n))$. To finish the proof it remains to let $\big((P_n-P_{n-1})(C[0,1])\big)$
be the desired FDD, where $P_n=\tilde Q_{\eta(n)}$.
\end{proof}
\end{lemma}
To simplify the notation, let us put $X=C[0,1]$ and let $P_n$ and $(a_n)$ be given by Lemma \ref{fddjohnson} for some fixed $\ep\in(0,1)$. Now define the function $\sigma:\N\rightarrow\N$, so that $\sigma(1)=1$ and $\sigma(n)=\sigma(n-1)+n$ for every $n\ge2$. Then, letting $Q_n=P_{\sigma(n)}$ and $X_n=(Q_n-Q_{n-1})(X)$ we have that $Y_n=span(\{a_{\sigma(n-1)+1},\dots,a_{\sigma(n)}\})$ is a subspace of $X_n$, $(1+\ep)$-isometric to $\ell_2^n$. From now on  in this section we are going to denote $E_n=Q_n(X).$\\

We define for every $n\in\N$ and $\delta>0$ the set 
$$B_n^\delta=\Big\{x\in X_n\;:\;d\big(x,B_{Y_n}\big)\le\delta\Big\}.$$
It is clear that $B_n^\delta$ is a compact convex subset of $(1+\delta)B_{X_n}$ that generates $X_n$. We will need the following  results to continue with our construction.

\begin{lemma}\label{lemmaballstospace}
Let $d$ be as in \eqref{dd} . For every $n\in\N$, if there is a Lipschitz retraction $\varphi:B_{E_n}\rightarrow B_{Y_n}$, then there is a Lipschitz retraction $\psi:\ell_{\infty}\rightarrow \ell_2^n$  satisfying the following inequality
$$||\psi||_{Lip}\le\frac{2d||\varphi||_{Lip}}{1-\ep}.$$
\begin{proof}
By composing $\varphi$ with the norm $2$ radial retraction from $E_n$ onto $B_{E_n}$ we have a retraction $\widehat{\varphi}: E_n\to B_{Y_n}$ with
$||\widehat{\varphi}||_{Lip}\le2||\varphi||_{Lip}$.
Using the technique in the proof of Theorem \ref{complemented}, since $Y_n$ is finite dimensional, there exists even a linear projection
 $\widetilde{\varphi}:E_n\to Y_n$ of norm $||\widetilde{\varphi}||\le2||\varphi||_{Lip}$.
Since $E_n$ is $d$-isometric to $\ell_\infty^M$, and $Y_n$ is $(1+\ep)$ isometric to $\ell_2^n$ the desired estimate follows.
\end{proof}
\end{lemma}

\begin{lemma}[Lindenstrauss,'64]\label{lemmaLind}
For every $n\in\N$, if $\psi:\ell_{\infty}\rightarrow\ell^{n}_2$ is a Lipschitz retraction then
$$||\psi||_{Lip}\ge \frac{n^{1/4}}{3}.$$
\begin{proof}
Just use Lemma 1.28 of \cite{BL2000} with $r=n^{1/4}$ and $\ep=1$.
\end{proof}
\end{lemma}

\begin{prop}\label{propbasecount}
There exist  sequences $(M_n),(\delta_n)\subset\R^+$, $M_n\to\infty,\;\delta_n\to0$, such that for every $n\in\N$ there is no $M_n$-Lipschitz retraction from $B_{E_n}$ onto $B_n^{\delta_n}$.
\begin{proof}
It suffices to let
$$M_n=\frac{n^{1/4}(1-\ep)}{25d}\;\;\;\;\forall n\in\N.$$
Now, for a fixed $n\in\N$ we are going to prove the existence of $\delta_n>0$ satisfying the statement of the Proposition by contradiction. Suppose that for every $\delta>0$ there exists a retraction $\phi^\delta:B_{E_n}\rightarrow B_n^{\delta}$ satisfying $||\phi^\delta||_{Lip}\le M_n$. Then we define $N:E_n\rightarrow B_{E_n}$ the radial projection, and $C^\delta:B_n^{\delta}\rightarrow B_{Y_n}$ a nearest point map. The map $\psi^{\delta}=C^\delta\circ \phi^\delta\circ N:E_n\rightarrow B_{Y_n}$ satisfies the following inequality
$$||\psi^{\delta}(x)-\psi^{\delta}(y)||\le2M_n\bigg(||x-y||+\frac{\delta}{M_n}\bigg),$$
so we are allowed to use the Proposition of \cite{B99}. Then, for every $\tau>0$ denoting by $\chi_\tau$ the indicator function of $\tau B_{E_n}$ we have that $\varphi^{\delta,\tau}=\psi^{\delta}*\chi_{\tau}:B_{E_n}\rightarrow B_{Y_n}$ is a Lipschitz map satisfying
$$||\varphi^{\delta,\tau}(x)-\psi^\delta(x)||\le2M_n\bigg(\tau+\frac{\delta}{M_n}\bigg)\;\;\;\forall x\in B_{E_n},$$
$$||\varphi^{\delta,\tau}||_{Lip}\le 2M_n\bigg( 1+\frac{\delta N_n}{2\tau M_n} \bigg).$$
Let us take a sequence $(\delta_k)\subset \R^+$ decreasing to 0, and put $\tau_k=\frac{\delta_kN_n}{2M_n}$ for every $k\in\N$. Denoting $\varphi_k=\varphi^{\delta_k,\tau_k}$ we have that $\varphi_k$ pointwise converge to a retraction $\varphi:B_{E_n}\rightarrow B_{Y_n}$ with the norm $||\varphi||_{Lip}\le 4M_n$. By Lemma \ref{lemmaballstospace} there exists a Lipschitz retraction $\psi:\ell_{\infty}\rightarrow \ell_p^n$ such that
$$||\psi||_{Lip}\le\frac{2d||\varphi||_{Lip}}{1-\ep}\le\frac{8dM_n}{1-\ep}=\frac{8}{25}n^{1/4}<\frac{n^{1/4}}{3}.$$
This contradicts Lindenstrauss' Lemma \ref{lemmaLind}.
\end{proof}
\end{prop}

Finally, we are ready to construct nonretractable convex subsets of $C[0,1]$. Note that if $\lambda_n$ below are decreasing sufficiently fast then the resulting
set $K$ will be small.

\begin{theorem}\label{counterex}
For every sequence $(\lambda_n)\subset \R^+$, the subset of $C[0,1]$ given by
$$K=\overline{\co}\bigg(\bigcup\limits_{n\in\N}\lambda_nB_n^{\delta_n}\bigg)$$
is not Lipschitz retractable.
\begin{proof}
Suppose there exists a Lipschitz retraction $\phi:C[0,1]\rightarrow K$ and let $||\phi||_{Lip}=L$. Then, there is an $n\in\N$ such that $L<\frac{M_n}{2M}$ where $M$ is the constant of the FDD $(Q_n)$, that is, $||Q_n||\le M$ for every $n\in\N$. Now, the retraction $R_n=(Q_n-Q_{n-1})\circ\restr{\phi}{E_n}:E_n\rightarrow \lambda_nB_n^{\delta_n}$ has norm $||R_n||_{Lip}\le2ML<M_n$. Finally, the retraction $F_n:E_n\rightarrow B_n^{\delta_n}$ given by $F_n(x)=\lambda_n^{-1}R_n(\lambda_nx)$ for every $x\in E_n$ has norm $||F_n||_{Lip}\le||R_n||_{Lip}<M_n$ which contradicts Proposition \ref{propbasecount}.
\end{proof}
\end{theorem}


\section{Nearest point map}

 The radial projection mapping from a Banach space $X$  onto $B_X$ can be shown to have a Lipschitz constant at most $2$, \cite{DW64}.

 Figueiredo and Karlovitz \cite{FK67} proved that a real normed linear space of dimension $3$ or higher is an inner product space
 if and only if the radial projection onto the unit ball is nonexpansive (see also  \cite{Phe58}, \cite{Sch65}).

For more general sets $C$ a natural candidate for the retraction mapping is the nearest point map (sometimes also called the proximity map,
of which the radial projection is a special case), provided the nearest point is unique. Again, it turns out that
such a mapping is nonexpansive for every closed convex set if and only if the space is Hilbertian \cite{Phe58}. In fact,
in the papers  \cite{BR74}
\cite{FK70}, the following characterization is shown.
 Let $X$ be a strictly convex Banach space with $3\le dim(X)$; then there exists a nonexpansive retraction onto the unit ball if and only if $X$ is a Hilbert space. 
The restriction to dimension at least $3$ is necessary.
Karlovitz \cite{Kar72} showed that if $X$ is $2$-dimensional then a non-expansive retraction onto any closed convex set always exists and it can be realized
 as a proximity mapping with respect to a new norm.

Let us begin our last section with
a new characterization of the Hilbert space in the spirit of the above results.

\begin{theorem}
Let $X$ be a Banach space. Then $X$ is isomorphic to a Hilbert space if and only if  every convex and compact subset $K\subset X$ is a Lipschitz retract
of $X$. 
\end{theorem}
\begin{proof}
Of course, it suffices to deal with infinite dimensional spaces $X$, and in light of the Phelps \cite{Phe58} characterization,
it suffices to prove only one implication (passing from retractions to Hilbert space).
By inspection of the proof of the Lindenstrauss-Tzafriri characterization of the Hilbert space \cite{LT71}, it is
clear that a Banach space is isomorphically Hilbert if and only if there exists some $\lambda>0$ such that every finite-dimensional
subspace $E\subset X$ is $\lambda$-complemented in $X$. 
Proceeding by contradiction, suppose that every convex compact is a Lipschitz retract of $X$, but the complementation constants $\lambda(E,X)$,
where $E$ runs through all finite-dimensional subspaces of $X$,
are not uniformly bounded. It is easy to see that this also means that for any finite-codimensional subspace $Y\subset X$, 
 the complementation constants $\lambda(E,X)$, where $E\subset Y$ is finite dimensional, are not uniformly bounded.
 We construct inductively a sequence $(E_n)$
of finite -dimensional subspaces of $X$, such that this sequence forms an FDD with the projection constant at most $2$ of its closed span
(a subspace of $X$), so that $\lambda(E_n,X)>n$. Let us describe the inductive step. Having constructed the initial sequence
$(E_n)_1^N$ such that the FDD has the projection constant bounded by $2-\frac1N$, we use the Mazur technique for constructing
basic sequences \cite{Fab1} p.191. Namely, there exists a finite set of functionals $\{f_i\}_1^M$ from $X^*$ that is $(1+\frac1{2N^2})$-norming
for $E_1\oplus\dots\oplus E_N$. Let $Y=\cap_1^M\text{Ker}f_i$ be a finite-codimensional subspace of $X$. Now, it suffices to choose
a finite dimensional subspace $E_{N+1}$ of $Y$ such that $\lambda(E_{N+1},X)>N$. The projection constant for the new FDD is under control.

By the same argument as in the proof of Theorem \ref{complemented}
 it follows that  the Lipschitz norm of any Lipschitz
retraction from $X$ onto $B_E$ cannot be less than $\lambda(E,X)$. Let us pass to the construction of $K$.
Fix a sequence $(\tau_n)$ of positive numbers such that $\sum \tau_n$ is finite.
Finally, set $K$ to be the closed convex hull of $\cup \tau_n B_{E_n}$. Supposing that $\phi:X\to K$ is a $\lambda$-Lipschitz retraction,
for some $\lambda>0$, we compose $\phi$ with the canonical retraction $r_n$ of $K$ onto $\tau_n B_{E_n}$ (which exists due to the
structure of $K$, and has a Lipschitz constant at most $6$), for $n>10\lambda$. Then $r_n\circ\phi:X\to\tau_n B_{E_n}$
is a Lipschitz retraction of Lipschitz norm at most $6\lambda$, hence
$\lambda(E_n,X)\le6\lambda$ which is a contradiction.
\end{proof}

This leaves us with the natural question whether at least uniformly continuous retractions onto convex compact subsets of a general Banach space are always possible.
We give a strong positive answer to this problem, showing that in fact under the URED renorming every convex and compact subset $K$ is
a uniform retract, from any bounded set $B$ containing $K$,  by means of the nearest point map. Without loss of generality, it suffices to deal with the case when $B=B_X$ is the unit ball of $X$.

\begin{definition}
Given a subset $K$ of a Banach space $X$, we will say that $X$ is $K$-URED if it is uniformly rotund in the direction $z$ for every $z\in span(K)$.
\end{definition}

\begin{definition}
Given a subset $K$ of a Banach space $X$, we will say that $X$ is $K$-UR if for every pair of sequences $(x_n),(y_n)\in S_X$ such that $x_n-y_n\in K$ and $||x_n+y_n||\rightarrow 2$ we have that $||x_n-y_n||\rightarrow 0$.
\end{definition}

\begin{lemma}\label{lemmamainUR}
Let $X$ be a Banach space that is uniformly rotund in the direction $z\in X\setminus\{0\}$. If there are $v_n,w_n\in S_X$ such that $v_n-w_n\rightarrow z$, then there exist $\widetilde{v_n},\widetilde{w_n}\in S_X$ such that $v_n-\widetilde{v_n},w_n-\widetilde{w_n}\rightarrow 0$ and $\widetilde{v_n}-\widetilde{w_n}=\lambda_nz$ for some $\lambda_n\in\R$.
\begin{proof}
First take for every $n\in\N$
$$t_n=\max\{t\ge0\;:\;w_n+tz\in S_X\}.$$
Let us prove that $t_n\rightarrow t\le1$. If we suppose that $t>1$, then taking $\lambda_n=1/t_n$ we have that
$$w_n+z=(1-\lambda_n)w_n+\lambda_n(w_n+t_nz),$$
where  $w_n,w_n+t_nz\in S_X$ and $w_n-(w_n+t_nz)=-t_nz$. In particular, $w_n+z\in B_X$  and we have that
$$1\ge||w_n+z||\ge |\,||v_n||-||z-(v_n-w_n)||\,|\rightarrow 1,$$
so by the assumption that the norm is uniformly rotund in the direction $z$ we get that $\lambda_n\rightarrow\lambda\in\{0,1\}$ which leads to a contradiction because $\lambda_n=1/t_n\rightarrow1/t\in(0,1)$. This means we can assume that $t\le1$.\\

If $t=1$ then we define $\widetilde{w_n}=w_n$ and $\widetilde{v_n}=w_n+t_nz$ and we are done. If $t<1$ we can assume that $t_n<1$ for every $n\in\N$. In this case, we define $\widetilde{v_n}=\frac{w_n+z}{||w_n+z||}$ and
$$s_n=\max\{s\ge0\;:\;\widetilde{v_n}+s(-z)\in S_X\}.$$
Following the same argument as above, we know that $s=\lim s_n\le1$. As $t<1$ we know that $||w_n+z||>1$, so $\frac{w_n}{||w_n+z||}=\widetilde{v_n}-\frac{z}{||w_n+z||}\in B_X$. Hence $s_n\ge \frac{1}{||w_n+z||}$. Just notice that $||w_n+z||\rightarrow 1$ because
$$1<||w_n+z||\le||v_n||+||z-(v_n-w_n)||\rightarrow 1,$$
so
$$1\ge s=\lim s_n\ge\lim\frac{1}{||w_n+z||}=1\;\;\;\Rightarrow\;\;\;s=1.$$
Finally, we are able to define $\widetilde{w_n}=\widetilde{v_n}+s_n(-z)$.
\end{proof}
\end{lemma}

\begin{prop}\label{equivURED}
Let $X$ be a Banach space and $z\in S_X$, the following assertions are equivalent:
\begin{itemize}
\item X is uniformly rotund in the direction $z$.
\item For every pair of sequences $x_n,y_n\in S_X$ such that $x_n-y_n\rightarrow \lambda z,$ for some $\lambda\in \R$, if $||x_n+y_n||\rightarrow 2$ then $\lambda=0$.
\end{itemize}
\begin{proof}
If $X$ is uniformly rotund in the direction $z$ and we suppose by contradiction that the second assertion does not hold, then there exist $x_n,y_n\in S_X$ such that $x_n-y_n\rightarrow \lambda z$ for some $\lambda\neq0$ with $||x_n+y_n||\rightarrow 2$. Then using Lemma \ref{lemmamainUR} we get $\widetilde{x_n},\widetilde{y_n}\in S_X$ such that $\widetilde{x_n}-x_n,\widetilde{y_n}-y_n\rightarrow 0$ and $\widetilde{x_n}-\widetilde{y_n}=\lambda_nz$. Now, we have
$$2\ge||\widetilde{x_n}+\widetilde{y_n}||\ge||x_n+y_n||-(||x_n-\widetilde{x_n}||+||y_n-\widetilde{y_n}||)\rightarrow 2,$$
so $\lambda_n\rightarrow0$ which is impossible because $\lambda_n z=\widetilde{x_n}-\widetilde{y_n}\rightarrow \lambda z\neq0$.
The other implication is straightforward.
\end{proof}
\end{prop}

\begin{prop}
Let $X$ be a Banach space and $K\subset X$. If $K$ is compact, then $X$ is $K$-URED if and only if it is $K$-UR.
\begin{proof}
If $X$ is $K$-URED, let us take $x_n,y_n\in S_X$ such that $x_n-y_n\in K$ and $||x_n+y_n||\rightarrow 2$. Then, as $K$ is compact, $x_n-y_n\rightarrow z\in K\subset span(K)$ so $||x_n-y_n||\rightarrow 0$ just by Proposition \ref{equivURED}.\\
If $X$ is $K$-UR then it is trivially $K-URED$.
\end{proof}
\end{prop}

\begin{prop}\label{mainpropUR}
Let $X$ be a Banach space and $K\subset B_X$ be a compact convex subset. If $X$ is $K-URED$ then the nearest point map from $B_X$ onto $K$ is uniformly continuous.
\begin{proof}
We are going to argue by contradiction.
Assuming that the nearest point map $R:B_X\rightarrow K$ is not uniformly continuous, then there exists an $\ep>0$ and a pair of sequences $(x_n)_{n\in\N},(y_n)_{n\in\N}\subset B_X$ such that $||x_n-y_n||\rightarrow 0$ and $||Rx_n-Ry_n||>\ep$ for every $n\in\N$.
By compactness, we may assume that the sequences $(||Rx_n-x_n||)$, $(||Ry_n-y_n||)$, $(Rx_n)$ and $(Ry_n)$ are all convergent. In particular, we claim that $\lim ||Rx_n-x_n||=\lim||Ry_n-y_n||=d\in\R^+$.
Let us first prove that $\lim ||Rx_n-x_n||=\lim||Ry_n-y_n||$. Otherwise, we may assume that $\lim||Rx_n-x_n||>\lim||Ry_n-y_n||+\rho$ for some $\rho>0$. Then there is an $n\in\N$ such that $||Rx_n-x_n||>||Ry_n-y_n||+\rho/2$ and $||x_n-y_n||\le\rho/2$, so
$$||Ry_n-x_n||\le||Ry_n-y_n||+||x_n-y_n||\le||Ry_n-y_n||+\rho/2<||Rx_n-x_n||.$$
This contradicts the definition of $R$. Now, if $d=0$ then $Rx_n,x_n\rightarrow p\in K$ and $Ry_n,y_n\rightarrow q\in K$ so $p=q$ because $x_n-y_n\rightarrow 0$. This means that $\lim||Rx_n-Ry_n||\rightarrow 0$ which is  impossible.\\

Assuming that $d>0$, we are going to use the previous Lemma \ref{lemmamainUR} with $v_n=\frac{Rx_n-x_n}{||Rx_n-x_n||}$, $w_n=\frac{Ry_n-y_n}{||Ry_n-y_n||}$ and $z=\frac{p-q}{d}$ where $p=\lim Rx_n$ and $q=\lim Ry_n$. Indeed,
$$v_n-w_n=\frac{Rx_n}{||Rx_n-x_n||}-\frac{Ry_n}{||Ry_n-y_n||}+\frac{y_n||Rx_n-x_n||-x_n||Ry_n-y_n||}{||Rx_n-x_n||\,||Ry_n-y_n||},$$
so $v_n-w_n\rightarrow z=\frac{p-q}{d}$. Let us take $\widetilde{v_n}, \widetilde{w_n}\in S_X$ given by Lemma \ref{lemmamainUR}. As
$$\left|\left| \widetilde{v_n}+\widetilde{w_n} \right|\right|\ge||v_n+w_n||-(||\widetilde{v_n}-v_n||+||\widetilde{w_n}-w_n||),$$
we have that $2\ge\lim||\widetilde{v_n}+\widetilde{w_n}||\ge\lim||v_n+w_n||$. For this reason, in order to prove that $||\widetilde{v_n}+\widetilde{w_n}||\rightarrow 2$, it is enough to prove that $\left|\left|\frac{{v_n}+{w_n}}{2} \right|\right|\rightarrow 1$.\\

Equivalently, we are going to prove that $\left|\left| \frac{(Rx_n-x_n)+(Ry_n-y_n)}{2} \right|\right|\rightarrow d$. If that does not hold, then there exists $\rho>0$ such that
$$d=\lim||Rx_n-x_n||>\lim\left|\left| \frac{(Rx_n-x_n)+(Ry_n-y_n)}{2} \right|\right|+\rho.$$
Now, there exists an $n\in\N$ such that 
$$||Rx_n-x_n||>\left|\left| \frac{(Rx_n-x_n)+(Ry_n-y_n)}{2} \right|\right|+\rho/2\;\;\;\;,\;\;\;\;||x_n-y_n||<\rho.$$
So taking into account the convexity of $K$, the next inequality is a contradiction with the definition of $R$:
$$\begin{aligned} \left|\left| \frac{Rx_n+Ry_n}{2}-x_n \right|\right|&\le \left|\left| \frac{(Rx_n-x_n)+(Ry_n-y_n)}{2} \right|\right|+\left|\left| \frac{x_n+y_n}{2}-x_n \right|\right|\\
&<\left|\left| \frac{(Rx_n-x_n)+(Ry_n-y_n)}{2} \right|\right|+\rho/2<||Rx_n-x_n||. \end{aligned}$$
Finally, we know that $||\widetilde{v_n}+\widetilde{w_n}||\rightarrow 2$ and $\widetilde{v_n}-\widetilde{w_n}=\lambda_n(p-q)$ with $\lambda_n\in\R$. As the norm of $X$ is $K$-URED, we get that $\lambda_n\rightarrow0$. This is in fact impossible because
$$||\widetilde{v_n}-\widetilde{w_n}||\ge||v_n-w_n||-(||\widetilde{v_n}-v_n||+||\widetilde{w_n}-w_n||)\rightarrow||z||>0.$$
\end{proof}
\end{prop}

Recall that thanks to Zizler's result  (\cite{Ziz1}) every separable Banach space has an equivalent URED renorming.
\begin{corollary}
If $X$ is a  separable Banach space, then it can be renormed so that the nearest point map from $B_X$ onto any compact and convex subset of $B_X$ is uniformly continuous.
\end{corollary}

The next proof is a slight modification of the proof given by V. Zizler to renorm separable spaces with URED norm (\cite{Ziz1}).

\begin{prop}
For every compact subset $K$ of a Banach space $X$, there is an equivalent $K$-URED norm on $X$.
\begin{proof}
Let us take $V=\overline{span(K)}$ which has to be a separable Banach space. Now, there is a countable set $\{f_n\}_{n\in\N}\subset S_{X^*}$ separating the points of $V$. We define the linear operator  $T:X\rightarrow\ell_2$ as $Tx=\left(\frac{f_n(x)}{2^n}\right)_{n\in\N}$. It is easily seen that $T$ is  bounded and its restriction to $V$ is injective. Our new equivalent norm is going to be
$$|||x|||=\sqrt{||x||_X^2+||Tx||_2^2},$$
where $||\cdot||_X$ is the initial norm on $X$ and $||\cdot||_2$ is the usual norm on $\ell_2$. If we suppose that there exists a bounded sequence $(x_n)\subset X$ such that
\begin{equation}\label{eqURED}2(|||x_n+z|||^2+|||x_n|||^2)-|||2x_n+z|||^2\rightarrow 0,\end{equation}
for some $z\in V$, then
$$2(||x_n+z||^2+||x_n||^2)-||2x_n+z||^2+2(||Tx_n+Tz||_2^2+||Tx_n||_2^2)-||2Tx_n+Tz||_2^2\rightarrow0.$$
In particular, $2(||Tx_n+Tz||_2^2+||Tx_n||_2^2)-||2Tx_n+Tz||_2^2\rightarrow0$. Making use of Proposition 1 of \cite{Ziz1}, the equivalence $1\Leftrightarrow7$ means that $\ell_2$ is not uniformly rotund in the direction $Tz$, leading to a contradiction because $\ell_2$ is in fact UR. Then, \eqref{eqURED} does not hold for any bounded sequence $(x_n)\subset X$ which again by the equivalence $1\Leftrightarrow7$ means that $X$ is uniformly rotund in the direction z.
\end{proof}
\end{prop}

\begin{corollary}
For every compact and convex subset $K\subset B_X$ of a Banach space $X$, there is an equivalent norm on X such that the nearest point map from $B_X$ onto $K$  is well defined and uniformly continuous. 
\end{corollary}

The following corollaries were recently established in \cite{CCW21}.
\begin{corollary}\label{corUC}
For every compact and convex subset $K$ of a Banach space $X$ there is a uniformly continuous retraction from X onto K.
\end{corollary}
\begin{proof}
We may assume without loss of generality that $K\subset B_X$.
It suffices to compose the $2$-Lipschitz retraction of $X$ onto $B_X$ with the uniformly continuous retraction from $B_X$ onto $K$ obtained earlier.
\end{proof}

\begin{corollary}
Compact convex subsets of Banach spaces are absolute uniform retracts.
\begin{proof}
Just embed the metric space into $\ell_\infty(\Gamma)$ for some set $\Gamma$ and use Corollary \ref{corUC}. 
\end{proof}
\end{corollary}

{\it \bf Acknowledgements.} We would like to thank Bill Johnson for generously providing us with the statement and proof of the apparently yet unpublished Lemma \ref{lemmaJohnson}, originally due to Vitali Milman, which  has significantly simplified and strengthened our original argument.
Our thanks are extended also to Vitali Milman, who has permitted us to include his important result in our note.

\bigskip

\printbibliography
\end{document}